\pdfoutput=1
\documentclass[psamsfonts]{amsart}

\usepackage{tikz-cd}
\usepackage{amssymb,amsfonts}
\usepackage{verbatim}
\usepackage{latexsym,amsthm,graphicx}
\usepackage{amsmath,amscd}
\usepackage[all,arc]{xy}
\usepackage{enumerate}
\usepackage{mathrsfs}
\usepackage{multirow}
\usepackage{times}
\usepackage{latexsym,amssymb,amsmath,graphicx, amsthm}
\usepackage[margin=1in,left=1in]{geometry}

\newtheorem{thm}{Theorem}[section]
\newtheorem{cor}[thm]{Corollary}
\newtheorem{prop}[thm]{Proposition}
\newtheorem{lem}[thm]{Lemma}

\def\be{\begin{eqnarray}}
\def\ee{\end{eqnarray}}

\def\b*{\begin{eqnarray*}}
\def\e*{\end{eqnarray*}}

\newcommand {\A}{\mathbb{A}}
\newcommand {\V}{\mathbb{V}}

\newcommand{\mld}{\widehat{\textnormal{mld}}}
\newcommand{\Spec}{\textnormal{Spec}}
\newcommand{\ord}{\textnormal{ord}}
\newcommand{\Int}{\textnormal{Int}}
\newcommand{\cont}{\textnormal{Cont}}

\theoremstyle{definition}
\newtheorem{defn}[thm]{Definition}

\newtheorem{exmp}[thm]{Example}

\theoremstyle{remark}
\newtheorem{rem}[thm]{Remark}
\newtheorem{rems}[thm]{Remarks}

\newtheorem{cond}[thm]{Condition $\Delta^\alpha $}

\makeatletter
\let\c@equation\c@thm
\makeatother
\numberwithin{equation}{section}

\title{Computations of Mather Minimal Log Discrepancies}

\author{Weichen Gu}

\date{July 9, 2017}

\begin{document}
\bibliographystyle{alpha}

\maketitle
\begin{abstract}
We compute the Mather minimal log discrepancy via jet schemes and arc spaces for toric varieties and very general hypersurfaces.
\end{abstract}


\section{Introduction}
\label{intro}

The minimal log discrepancy is an important invariant in algebraic geometry. It is well known that certain conjectures on the minimal log discrepancy imply the termination of the Minimal Model Program (see \cite{Sh04}). However, not much about minimal log discrepancy is known compared to other invariants defined in similar settings such as the log canonical threshold. Recently, the notion of \emph{Mather minimal log discrepancy} was introduced by Ishii in \cite{Ish11}. It is closely related to the minimal log discrepancy and they share many similar properties. But the Mather minimal log discrepancy is defined more generally for an arbitrary variety. This paper is concerned with the computation of Mather minimal log discrepancy in the context of toric varieties and very general hypersurfaces.

Let us start by recalling the definition of the minimal log discrepancy. Let $X$ be a normal $\mathbb{Q}$-Gorenstein variety over an algebraically closed field $k$ of characteristic zero and let $f: Y\rightarrow X$ be a birational morphism with $Y$ normal. For a divisor $E$ on $Y$ over $X$ and an ideal $\mathfrak{a}$ in $\mathcal{O}_X$, the \emph{log discrepancy} of the pair $(X,\mathfrak{a})$ with respect to $E$ is defined as
\begin{equation*}
a(E;X,\mathfrak{a}):=\ord_E(K_{Y/ X})-\ord_E(\mathfrak{a})+1
\end{equation*}
For each closed subset $W$ in $X$, the minimal log discrepancy of $(X,\mathfrak{a})$ with respect to $W$ is
\begin{equation*}
\textnormal{mld}(W;X,\mathfrak{a}):=\min\{a(E;X,\mathfrak{a}) | c_X(E)\subset W\},
\end{equation*}
where $c_X(E)$ is the center of $E$ on $X$. An introduction to minimal log discrepancies can be found in \cite{Am06}.

Now let $X$ be an arbitrary variety over an algebraically closed field $k$ of characteristic zero. Let $f:Y\rightarrow X$ be a resolution of singularities so that $Y$ is a sufficiently "high" birational model over $X$ (will be made clear in Section \ref{sec3}). The Mather minimal log discrepancy is defined in a similar way to the usual minimal log discrepancy (by simply replacing the relative canonical divisor with the \emph{Mather discrepancy divisor}) but it is much easier to describe in terms of jet schemes and arc spaces. The Mather minimal log discrepancy for a closed point $x$ of a variety $X$ is denoted by $\mld(x;X)$. Of the many nice properties of the Mather minimal log discrepancy, one of the most important is Inversion of Adjunction (\cite[Theorem 4.10]{dFD11} and \cite[Proposition 3.10]{Ish11}).

When both Mather and usual minimal log discrepancies are defined, the two differ by the pull back of a certain ideal sheaf (\cite[2.2]{Ish11}). In particular, Mather minimal log discrepancy is always larger than or equal to the usual minimal log discrepancy. Their relation has been further studied in \cite{Ish13}, \cite{EI13} and \cite{dFT16}. We note that contrary to usual minimal log discrepancies, the variety has "good" singularities when Mather minimal log discrepancies are small (see \cite[Theorem 4.7]{Ish11} for a more precise description).

In Section \ref{sec2} we give an overview of jet schemes and arc spaces. We start with the definition of jet schemes, and apply the definition to describing jet schemes of an affine variety over a field. It shows that the jet schemes of an affine variety $X$ are also affine and we get explicit defining equations for $X_m$. This explicit description will be important for our analysis in Section \ref{sec5}. Next we review arc spaces and cylinders (especially contact loci). The arc space of a variety $X$, denoted by $X_\infty$, is the projective limit of the projective system $\{X_m\}_{0\leq m<\infty}$ of jet schemes, and cylinders are inverse images of constructible subsets of $X_m$ in $X_\infty$.

In Section \ref{sec3} we review the basics about Mather minimal log discrepancy. We start by defining the notion of Mather discrepancy divisor through Nash blow-ups. Then we recall the following result connecting the Mather minimal log discrepancy to jet schemes:

\begin{prop}
\label{fundamental_prop1}
(\cite[Lemma 4.2]{Ish11})
Let $X$ be a variety over an algebraically closed field $k$ of characteristic $0$. If $x$ is a closed point of $X$, then we have
\begin{equation*}
\mld(x;X)=\lim_{m\rightarrow \infty}((m+1)\dim(X)-\dim (\psi_m(\pi^{-1}(x)))),
\end{equation*}
where $\psi_m: X_\infty\rightarrow X_m$ and $\pi: X_\infty\rightarrow X$ are canonical truncation maps.
\end{prop}

The key point for the examples considered in Section \ref{sec4} and Section \ref{sec5}, is to compute/bound $\dim(\psi_m(\pi^{-1}(x)))$ for $m$ large enough.

Section \ref{sec4} is devoted to the study of the Mather minimal log discrepancy for a toric variety at a closed point $x$. The question is local so we assume $X=X(\sigma)$ is the affine toric variety associated to the cone $\sigma\subset N_\mathbb{R}:= N\otimes_\mathbb{Z}\mathbb{R}$, where $N\cong\mathbb{Z}^n$ is the lattice of $\sigma$. We further assume that $\sigma$ spans $N_\mathbb{R}$. First, we consider the case when $x$ is a torus-invariant point. By Proposition \ref{fundamental_prop1}, the key is to compute $\dim(\psi_m(\pi^{-1}(x)))$ for $m$ large enough. The space $\psi_m(\pi^{-1}(x))$ is decomposed into $T_m$-orbits, where $T_m$ is the $m^{\textnormal{th}}$ jet scheme of the torus $T$ in $X$ which naturally acts on $X_m$. We use the fact that those orbits correspond to lattice points in the interior of $\sigma$. This characterization of orbits follows from the work of Ishii (\cite{Ish03}). The problem thus comes down to finding the dimension of each $T_m$-orbit, which is in turn done by computing the dimension of its stabilizer.

In order to state our result, we introduce some notation. Let $n$ be the dimension of $X$ and $M=N^\vee$ be the dual lattice. We define the dual space $M_\mathbb{R}:=M\otimes_\mathbb{Z} \mathbb{R}$ and the dual cone $\sigma^\vee:=\{u\in M_\mathbb{R} | \langle u,v\rangle\geq 0\textnormal{ for all }v\in \sigma\}$. With this notation, we show the dimension of the $T_m$-orbit associated to a lattice point $a$ in the interior of $\sigma$ is equal to
\begin{equation*}
(m+1)n- \min \big\{\sum_{i=1}^n \langle a,u_{i} \rangle | u_{1},\ldots,u_{n}\ \textnormal{span}\  M_\mathbb{R},\textnormal{ with }u_i\in M\cap \sigma^\vee\textnormal{ for each }i\big\},
\end{equation*}
where the minimum is run over all linearly independent sets of vectors $\{u_1,\ldots,u_n\}$ in $M\cap \sigma^\vee$. Now we just need to let the point $a$ vary and take the maximum. Hence we get the following theorem:

\begin{thm}
Let $X$ be an affine toric variety associated to a cone $\sigma$ of dimension $n$ over an algebraically closed field $k$ of characteristic zero. Let $N$ be the lattice of $\sigma$ and $M$ be the dual lattice. If $\sigma$ spans $N_\mathbb{R}$ and $x$ is the torus-invariant point, then we have
\begin{equation*}
\mld (x;X)=\min_{a\in \Int(\sigma)\cap N}\Big\{ \min \big\{\sum_{i=1}^n \langle a,u_{i} \rangle | u_{1},\ldots,u_{n}\ \textnormal{span}\  M_\mathbb{R},\ u_i\in M\cap \sigma^\vee\textnormal{ for each }i\big\}\Big\},
\end{equation*}
where the second minimum is taken over all linearly independent sets of vectors $\{u_1,\ldots,u_n\}$ in $M\cap \sigma^\vee$.
\end{thm}

We use the theorem to compute $\mld(x;X)$ in some examples. For example, we show that if $X$ is a toric surface, then $\mld (x;X)=\dim(X)$ (which is $2$). In higher dimension, the same conclusion holds if the torus-fixed point $x$ is an isolated singularity point and $X$ is simplicial. We also give some examples where $\mld(x;X)\neq \dim(X)$.

We conclude Section \ref{sec4} by considering an arbitrary closed point on a toric variety $X$. Recall that the set of closed points of $X$ is a disjoint union of $T$-orbits associated to faces of the cone $\sigma$. Each orbit is generated by a distinguished point associated to the corresponding face. Therefore, the problem reduces to computing the Mather minimal log discrepancy at these distinguished points, and it is further reduced to the case of a torus-invariant point in the following sense:

\begin{thm}
Let $X=X(\sigma)$ be an affine toric variety of dimension $n$ over an algebraically closed field $k$ of characteristic zero. Let $\tau$ be a face of $\sigma$ of dimension $k<n$ and $x_\tau$ be the distinguished point associated to $\tau$. If $Y$ is the $k$-dimensional affine toric variety associated to the cone $\tau$ and $y$ is the torus-invariant point of $Y$, then we have
\begin{equation*}
\mld (x_\tau;X)-n=\mld (y;Y)-k.
\end{equation*}
\end{thm}

We consider the case of very general hypersurfaces in Section \ref{sec5}. Let $f=\sum_{i=1}^N a_{I^i} x^{I^i}$ be the defining equation of a hypersurface $X\subset \mathbb{A}^{n+1}$, where $I^i$ are multi-indices and $x^{I^i}$ stands for $\Pi _{j=1}^{n+1} x_j ^{I_j^i}$. The \emph{support} of $f$ is the set $A:=\{I^1,\ldots,I^N\}\subset \mathbb{Z}^{n+1}$. When $A\neq \emptyset$, the \emph{dimension} of $A$ is the dimension of the linear span over $\mathbb{Q}$ of the convex hull of $A-a$, for any $a\in A$. Following from the result of Yu (\cite[Theorem 3]{Yu16}), we deduce that for a support $A$ such that $\dim(A)\geq 2$ or $\dim(A)=1$ and the convex hull of $A$ contains exactly two integral points, and for general coefficients $a_{I^i}$, $X$ is an integral hypersurface. Then, under a certain generality condition, we give a lower bound for the Mather minimal log discrepancy of $X$ at the origin. As in the case of toric varieties, we write $\psi_m(\pi^{-1}(0))$ as a disjoint union, up to the image of a thin set, of subsets $C^m_\alpha$, with $\alpha=(\alpha_1,\ldots,\alpha_{n+1})$ running over all $(n+1)$-tuples of positive integers. For simplicity, we define the \emph{product} of an $(n+1)$-tuple $\alpha$ with a multi-index $I$ as $\alpha\cdot I:=\sum_{j=1}^{n+1}\alpha_j I_j$. An $(n+1)$-tuple $\alpha$ is called \emph{feasible} if $\min_{1\leq i \leq N} \{\alpha \cdot I^i\}$ is attained by at least two different $i$'s. We show that $C^m_\alpha=\emptyset$ if $\alpha$ is not feasible; when $f$ has a fixed support and very general coefficients, $\dim(C^m_\alpha)$ is bounded above by
\begin{equation*}
mn-\sum_{j=1}^{n+1} (\alpha_j-1)-1+ \min_{1\leq i\leq N} \{I^i\cdot \alpha \}-\underset{1\leq j\leq n+1\textnormal{ with }  I_j^i>0}{\min_{1\leq i\leq N}}\{I^i\cdot \alpha-\alpha_j \}.
\end{equation*}
By taking the maximum over all feasible $\alpha$'s, we obtain the following theorem:

\begin{thm}
Let $f=\sum _{i=1} ^N a_{I^i} x^{I^i}$ be a polynomial with a fixed support $A$ such that $f$ has no constant term and that $f$ is not divisible by any $x_i$, and let $X$ be the hypersurface defined by $f$. If $A$ is $1$-dimensional and the convex hull contains only two integral points, or if $A$ has dimension $\geq 2$, then for very general coefficients $(a_{I^i})_{1\leq i\leq N}$, the hypersurface $X$ is integral and we have
\begin{equation*}
\mld(0;X)\geq \underset{\alpha}{\min} \{\sum_{j=1}^{n+1} (\alpha_j-1)+1-\underset{1\leq i\leq N}{\min} \{I^i\cdot \alpha \}+\underset{1\leq j\leq n+1\textnormal{ with }  I_j^i>0}{\min_{1\leq i\leq N}}\{I^i\cdot \alpha-\alpha_j \} \}+n,
\end{equation*}
where the first minimum is taken over all feasible $(n+1)$-tuples $\alpha$.
\end{thm}

In spite of the fact that the theorem only gives a lower bound of Mather minimal log discrepancy, we can use the proof of the above result to show that the inequality is actually an equality in many cases. We end the section with various examples. These examples show that the lower bound can be attained in many cases, but we also see that the inequality in the theorem can be strict.

\proof[Acknowledgements]
I want to express my deepest gratitude to my advisor, Mircea Musta\c{t}\u{a}, for years of guidance and patience. This paper would not have been possible without him. I also thank Karen Smith for a lot of helpful comments, and Mattias Jonsson, James Tappenden and Yuchen Zhang for some corrections.

\section{Preliminaries on jet schemes and arcs spaces}
\label{sec2}
In this section we review some basic properties of jet schemes and arc spaces that we need in the following sections. We mostly follow \cite{EM09}. For more details, see \cite{M14}, \cite{DL99} and \cite{dF16}.

\subsection{Jet schemes}
$\\$
A variety is an integral, separated scheme of finite type over a field. Let $k$ be an algebraically closed field of arbitrary characteristic and $X$ be a scheme of finite type over $k$. For each nonnegative integer $m$, we define the $m^{\textnormal{th}}$ \emph{jet scheme} of $X$, denoted by $X_m$, to be a scheme over $k$ such that for every $k$-algebra $A$ we have a functorial bijection
\begin{equation}
\label{eqn_jet_defn}
\textnormal{Hom}_{\textnormal{Sch}/k}(\Spec(A),X_m)\cong \textnormal{Hom}_{\textnormal{Sch}/k}(\Spec\ A[t]/(t^{m+1}),X).
\end{equation}

The jet schemes $X_m$ exist according to \cite[Proposition 2.2]{EM09}. Moreover, they are unique up to a canonical isomorphism since the bijection (\ref{eqn_jet_defn}) describes the functor of points of $X_m$. In particular, each element of the left-hand side of the bijection (\ref{eqn_jet_defn}) is an $A$-valued point of $X_m$, which is also called an $A$-valued $m$-jet of $X$. A $k$-valued point of $X_m$ is simply called an $m$-jet of $X$. Clearly when $m=0$ we have $X_0\cong X$. The canonical truncation map $A[t]/(t^{m+1})\rightarrow A[t]/(t^{p+1})$ for $m>p$ induces the map
\begin{equation*}
\textnormal{Hom}(\Spec\ A[t]/(t^{m+1}),X)\longrightarrow \textnormal{Hom}(\Spec\ A[t]/(t^{p+1}),X).
\end{equation*}
This induces via the bijection (\ref{eqn_jet_defn}) a canonical projection $\pi_{m,p}:X_m\longrightarrow X_p$. We denote this map by $\pi_m$ when $p=0$. These canonical projections satisfy the obvious compatibilities $\pi_{p,q}\circ \pi_{m,p}=\pi_{m,q}$ for $m>p>q$.

\begin{rems}
\label{jet_rems}
The following facts follow easily from the definition:

(i) If $f:X\rightarrow Y$ is a morphism of schemes of finite type over $k$, then there is an induced morphism of jet schemes $f_m:X_m\rightarrow Y_m$. Note that the induced maps $f_m$ are compatible with the canonical projections $\pi_{p,q}$, i.e. $\pi_{m,p}\circ f_m=f_p\circ \pi_{m,p}$.

(ii) For schemes $X$ and $Y$ of finite type over $k$, there is a canonical isomorphism
\begin{equation*}
(X\times Y)_m\cong X_m\times Y_m,
\end{equation*}
for every $m\geq 0$.

(iii) If $G$ is a group scheme over $k$ acting on a scheme $X$ of finite type over $k$, then $G_m$ is also a group scheme over $k$ and it acts on $X_m$.
\end{rems}

\begin{exmp}
Consider an affine scheme $X\hookrightarrow \A^n$ and let $g_1,\ldots,g_r\in k[x_1,\ldots,x_n]$ be generators for the ideal defining $X$. For a $k$-algebra $A$, consider an $A$-valued m-jet $\gamma$ of $X$ represented by $\gamma :\Spec\ A[t]/(t^{m+1})\rightarrow X$. Giving $\gamma$ is equivalent to giving a morphism of $k$-algebras
\begin{equation*}
\gamma^\ast: k[x_1,\ldots,x_n]/[g_1,\ldots,g_r]\longrightarrow A[t]/(t^{m+1}).
\end{equation*}
Let us write
\begin{equation*}
\gamma^\ast(x_i)=\sum_{j=0}^m x_i^{(j)} t^j, \textnormal{ for } 1\leq i\leq n.
\end{equation*}
They should satisfy $g_l(\gamma^\ast(x_1),\ldots,\gamma^\ast(x_n))=0$ in $k[t]\slash t^{m+1}$ for $1\leq l\leq r$. If we write
\begin{equation*}
g_l(\sum_{j=0}^m x_1^{(j)} t^j,\ldots, \sum_{j=0}^m x_n^{(j)} t^j)=\sum_{j=0}^m G_l^{(j)}(\underline{x})t^j\quad (\textnormal{mod } t^{m+1}),
\end{equation*}
we see that
\begin{equation}
X_m\cong \Spec\ k[x_i^{(j)} | 1\leq i\leq n, 1\leq j\leq m]/(G_l^{(j)}|1\leq l\leq r, 0\leq j\leq m).
\label{jet_affine}
\end{equation}
In particular, we conclude the jet schemes of an affine scheme are also affine schemes, of finite type over $k$.
\end{exmp}

\begin{rem}
It follows from the above example that the canonical projections $\pi_{m,p}:X_m\rightarrow X_p$ are affine morphisms.
\end{rem}

\begin{rem}
Another consequence of the above example is that if $X\hookrightarrow \A^n$ is a closed immersion, then the induced morphism of jet schemes $X_m\hookrightarrow (\A^n)_m$ is also a closed immersion. Moreover, we deduce from the explicit description of the equations of $X_m$ in $(\A^n)_m$ that more generally, if $X\hookrightarrow Y$ is a closed immersion then so is the induced map $X_m\rightarrow Y_m$.
\end{rem}

\begin{exmp}
\label{jet_affine_space}
The simplest (but important) example is $X=\A^n$. It follows immediately from equation (\ref{jet_affine}) that $(\A^n)_m\cong \A^{(m+1)n}$. Furthermore, the canonical projections $\pi_{m,p}$ are just projections along certain coordinate planes.
\end{exmp}

\begin{lem}
\label{lem_etale}
(\cite[Lemma 2.9]{EM09}) If $f:X\rightarrow Y$ is an \'etale morphism, then for every $m\geq 0$ the following commutative diagram is Cartesian:

$$\begin{CD}
X_m @>f_m>> Y_m\\
@V\pi_m^X VV @V\pi_m^Y VV\\
X @>f>> Y.
\end{CD}$$
\end{lem}

\begin{cor}
\label{jet_smooth}
(\cite[Corollary 2.11]{EM09}) If $X$ is a smooth variety of dimension $n$, then the canonical projections $\pi_{m,p}$ are locally trivial fibrations with fiber $\mathbb{A}^{(m-p)n}$. In particular, $X_m$ is smooth of dimension $(m+1)n$.
\end{cor}

\begin{proof}
For every point $x\in X$, one can find an open subset $x\in U$ and an \'etale morphism $U\rightarrow \mathbb{A}^n$. Using Lemma \ref{lem_etale}, the assertion is reduced to the case of an affine space, which follows from Example \ref{jet_affine_space}.
\end{proof}

\subsection{Arc spaces and cylinders}
\label{section_arc_spaces}
$\\$
Given a scheme $X$ of finite type over $k$ as before, we have a projective system
\begin{equation*}
\dots \longrightarrow X_m\longrightarrow X_{m-1}\longrightarrow \dots\longrightarrow X_1\longrightarrow X_0=X,
\end{equation*}
in which all morphisms are affine. Therefore, the projective limit exists in the category of $k$-schemes. The projective limit is denoted by $X_\infty$ and it is called the \emph{arc space} of $X$. Unlike the jet schemes, the arc space is typically not of finite type over $k$. We denote by $\psi_m$ the canonical map $X_\infty\rightarrow X_m$. We also write $\pi:=\psi_0: X_\infty\rightarrow X_0=X$ for the projection to the original scheme $X$.

It follows from the definition of jet schemes and projective limit that for every field extension $K$ of $k$, we have functorial isomorphisms
\begin{equation*}
\textnormal{Hom}(\Spec(K), X_\infty)\cong \underset{\longleftarrow}{\lim}\ \textnormal{Hom}(\Spec\ K[t]/t^{m+1},X) \cong \textnormal{Hom}(\Spec\ K[\![t]\!], X).
\end{equation*}
A $k$-valued point of $X_\infty$ is called an arc on $X$ and is represented by
\begin{equation}
\label{arc}
\gamma:\Spec\ k[\![t]\!]\longrightarrow X.
\end{equation}
For every field extension $K$ of $k$, a $K$-valued point of $X_\infty$ is called an $K$-valued arc of $X$. From now on, whenever we deal with $X_m$ and $X_\infty$ we will restrict to their $k$-valued points. Since the jet schemes are of finite type over $k$ this causes no ambiguity. Note that since we only consider the $k$-valued points, $X_\infty$ is the set-theoretic projective limit of the $X_m$ and the Zariski topology on $X_\infty$ is the projective limit topology.
\begin{rem}
As in the case of jet schemes, if $f:X\rightarrow Y$ is a morphism of schemes of finite type over $k$, then we have an induced map on the arc spaces $f_\infty: X_\infty\rightarrow Y_\infty$ that is compatible with canonical projections.
\end{rem}

\begin{rem}
\label{arc_product}
For schemes $X$ and $Y$ of finite type over $k$, there is a canonical isomorphism $(X\times Y)_\infty\cong X_\infty\times Y_\infty$ and we have the following commutative diagram:
$$\begin{CD}
(X\times Y)_\infty @>\cong>> X_\infty\times Y_\infty\\
@V\psi_m^{X\times Y} VV @V\psi_m^X\times \psi_m^Y VV\\
(X\times Y)_m @>\cong>> X_m\times Y_m.
\end{CD}$$
\end{rem}

We now define the notion of cylinders. Recall that a \emph{constructible set} in a scheme of finite type over $k$ is a finite union of locally closed subsets. A \emph{cylinder} in $X_\infty$ is a subset of the form $C=\psi_m^{-1} (S)$, for some nonnegative integer $m$ and some constructible subset $S$ of $X_m$. The arc spaces are typically not of finite type over $k$. So far most study on arc spaces has been focusing on cylinders and their irreducible components.

There is a special type of cylinders, the \emph{contact loci}, that will play an important role in what follows. To an ideal sheaf $\mathfrak{a}$, we associate subsets of arcs with prescribed vanishing order along $\mathfrak{a}$. More precisely, if $\gamma: \Spec\ k[\![t]\!]\rightarrow X$ is an arc, the inverse image of $\mathfrak{a}$ is an ideal in $k[\![t]\!]$ generated by $t^r$, for some $r$ (if the ideal is not zero). This $r$ is \emph{the order of $\gamma$ along $\mathfrak{a}$}, denoted by $\ord_\gamma (\mathfrak{a})$. When the inverse image is zero, we put $\ord_\gamma(\mathfrak{a})=\infty$. A contact locus is a subset of $X_\infty$ of one of the following forms:
\begin{equation*}
\cont^e(\mathfrak{a}):=\{\gamma\in X_\infty | \ord_\gamma(\mathfrak{a})=e\},
\end{equation*}
or
\begin{equation*}
\cont^{\geq e}(\mathfrak{a}):=\{\gamma\in X_\infty | \ord_\gamma(\mathfrak{a})\geq e\}.
\end{equation*}
We can similarly define subsets of $X_m$ with specified order along $\mathfrak{a}$, namely $\cont^e(\mathfrak{a})_m$ and $\cont^{\geq e}(\mathfrak{a})_m$, for $m\geq e$. It is clear that for every $m\geq e$, we have
\begin{equation*}
\cont^e(\mathfrak{a})=\psi_m^{-1}(\cont^e(\mathfrak{a})_m),\ \cont^{\geq e}(\mathfrak{a})=\psi_m^{-1}(\cont^{\geq e}(\mathfrak{a})_m).
\end{equation*}
This implies that $\cont^e(\mathfrak{a})$ is a locally closed set and $\cont^{\geq e}(\mathfrak{a})$ is a closed set.

\begin{defn}
Let $X$ be a scheme of finite type over $k$ of pure dimension $d$. A subset $A\subset X_\infty$ is \emph{thin} if there is some closed subvariety $S$ of $X$ whose dimension is strictly less than $d$ such that $A\subset S_\infty$. If a subset $A$ is not thin, it is \emph{fat}.
\end{defn}

We need the following result for our discussion in the following chapters:
\begin{lem}
\label{lem_4.3}
(\cite[Proposition 5.10]{EM05})
Let $X$ be a variety over $k$ of dimension $d$. Then

(1) For every $m\geq 0$, we have
\begin{equation*}
\dim(\psi_m(X_\infty))\leq (m+1)d.
\end{equation*}

(2) For every $m,\ n\geq 0$ with $m\geq n$, the fibers of $\psi_m(X_\infty)\rightarrow \psi_n(X_\infty)$ are of dimension $\leq (m-n)d$.
\end{lem}

\section{Preliminaries on Mather minimal log discrepancy}
\label{sec3}
In this section we introduce the Mather minimal log discrepancy following \cite{Ish11}. The definition is very similar to the usual minimal log discrepancy. Details on usual minimal log discrepancy and its relation to arc spaces can be found in \cite{EMY03}. Results on the relation between Mather minimal log discrepancy and the usual minimal log discrepancy can be found in \cite{EI13} and \cite{Ish13}. For details on Mather minimal log discrepancy, we refer to \cite{Ish11}, \cite{EI13} and \cite{Ish15}.

\subsection{Definition}
\begin{defn}
Let $X$ be a variety over a field $k$ and $f:Y\rightarrow X$ be a proper birational morphism of varieties, with $Y$ normal. Each prime divisor $E$ on $Y$ gives a valuation $\textnormal{ord}_E$ on $K(Y)=K(X)$. Here $E$ is called a \emph{divisor over} $X$ and we equate two divisors on two normal varieties over $X$ if they give rise to the same valuation on $X$. The \emph{center} of $E$ is the closure of the image of $E$ on $X$. A \emph{divisorial valuation} on $X$ is one of the form $v=q\cdot \textnormal{ord}_E$ where $q$ is a positive integer and $E$ is a divisor over $X$.
\end{defn}

Let $X$ be a variety of dimension $d$ over an algebraically closed field $k$ of characteristic zero. For simplicity we write $\Omega_X$ for the sheaf of relative differentials $\Omega_{X/k}$. The projection
\begin{equation*}
\pi: \mathbb{P}_X(\wedge^d \Omega_X)\longrightarrow X
\end{equation*}
is an isomorphism over the smooth locus $X_{\textnormal{reg}}\subset X$. In particular, there is a section $\sigma: X_{\textnormal{reg}}\rightarrow \mathbb{P}_X(\wedge^d \Omega_X)$.

\begin{defn}
The \emph{Nash blow-up} of $X$ is the closure of the image of $\sigma$, and is denoted by $\hat{X}$. It is a variety over $k$ with a projective morphism $\pi|_{\hat{X}}: \hat{X}\rightarrow X$ that is an isomorphism over the smooth locus of $X$. The line bundle
\begin{equation*}
\widehat{K}_X:=\mathcal{O}_{\mathbb{P}_X(\wedge^d \Omega_X)}(1)|_{\hat{X}}
\end{equation*}
is called the \emph{Mather canonical line bundle} of $X$.
\end{defn}

\begin{rem}
If $X$ is smooth, then clearly $\hat{X}=X$ and $\widehat{K}_X$ is just the canonical line bundle of $X$. More generally, the Nash blow-up can be thought of as the parameter space of limits of all tangent directions at smooth points of $X$.
\end{rem}

One can always find a resolution of singularities $f:Y\rightarrow X$ that factors through the Nash blow-up. Then the image of the $f^\ast (\wedge ^d \Omega_X)$ under the canonical homomorphism
\begin{equation*}
\wedge^d df: f^\ast (\wedge ^d \Omega_X)\rightarrow \wedge^d\Omega_Y
\end{equation*}
is of the form $J\wedge^d \Omega_Y$ where $J$ is an invertible ideal sheaf on $Y$ (\cite[Proposition 1.7]{dFEI08}). Let $\widehat{K}_{Y/X}$ be the effective divisor defined by $J$. This is supported on the exceptional locus of $f$ and it is called the \emph{Mather discrepancy divisor}. For each prime divisor $E$ on $Y$, we define $\hat{k}_E:=\ord_E(\widehat{K}_{Y/X})$. If $v=q\cdot \ord_E$ is a divisorial valuation, we write $\hat{k}_v:=q\cdot\hat{k}_E$.

\begin{defn}
Let $(X,\mathfrak{a})$ be a pair where $X$ is a variety over $k$ and $\mathfrak{a}$ is a nonzero ideal in $\mathcal{O}_X$. For a closed subset $W$ of $X$, the Mather minimal log discrepancy of $(X,\mathfrak{a})$ along $W$ is defined as
\begin{equation*}
\mld(W;X,\mathfrak{a}):=\inf\{\hat{k}_E-\ord_E(\mathfrak{a})+1 | E\textnormal{ is a divisor over }X\textnormal{ with center in }W\}.
\end{equation*}
When $\dim(X)=1$ and the infimum is negative, we make the convention that
\begin{equation*}
\mld(W;X,\mathfrak{a})=-\infty.
\end{equation*}
\end{defn}

\begin{rem}
If $\dim(X)\geq 2$ and $\mld(W;X,\mathfrak{a})<0$, then $\mld(W;X,\mathfrak{a})=-\infty$ (see \cite[Remark 3.4]{Ish11}). This is why we make the convention for the case when $\dim(X)=1$.
\end{rem}

\subsection{Relation to jet schemes and arc spaces}
$\\$
From now on we specialize to the case when $W=\{x\}$ for some closed point $x\in X$ and $\mathfrak{a}=\mathcal{O}_X$. We denote the Mather minimal log discrepancy by $\mld(x;X)$ for simplicity and write $\dim(X)=d$.

\begin{defn}
If $X$ and $Y$ are varieties over $k$, and $A\subset X$ and $B\subset Y$ are constructible subsets. Then a map $f: A\rightarrow B$ is a \emph{piecewise trivial fibration with fiber} $F$, if there exists a finite partition of B into locally closed subsets $S$ of $Y$ such that $f^{-1}(S)$ is isomorphic to $S\times F$ and $f|_{f^{-1}(S)}$ is the projection $S\times F\rightarrow S$ under the isomorphism.

Recall that for a scheme $X$ of finite type over $k$, there are canonical morphisms $\pi:X_\infty\rightarrow X$ and $\psi_m:X_\infty\rightarrow X_m$ for every $m\geq 0$ (in Subsection \ref{section_arc_spaces}). 
\end{defn}

\begin{defn}
Fix a closed point $x$ of $X$. For every $m\geq 0$ we define
\begin{equation*}
\lambda_m(x):=md-\dim \psi_m(\pi^{-1}(x)).
\end{equation*}
When there is no confusion we simply write $\lambda_m$ instead of $\lambda_m(x)$.
\end{defn}

\begin{rem}
\label{rem_smooth}
Corollary \ref{jet_smooth} shows that when $X$ is a smooth variety and $x$ is a closed point of $X$, we have for each $m\geq 0$,
\begin{equation*}
\lambda_m(x)=m\dim(X)-\dim(\psi_m(\pi^{-1}(x)))=0.
\end{equation*}
\end{rem}

\begin{lem}
\label{lem_lambda}
(\cite[Lemma 4.2]{Ish11}) For every $m\geq 0$, we have $\lambda_m\geq 0$ and $\lambda_{m+1} \geq \lambda_m$. Moreover, $\lambda_m$ is constant for $m\gg 0$.
\end{lem}

\begin{defn}
\label{defn_lambda}
According to the lemma above, $\underset{m\rightarrow \infty}{\lim} \lambda_m(x)$ exists and it is equal to $\lambda_m(x)$ for all $m$ large enough. We denote this limit by $\lambda(x)$. When there is no confusion, we also write this limit as $\lambda$.
\end{defn}

The following result from \cite{Ish11} describes the Mather minimal log discrepancy in terms of jet schemes and arc spaces. We use this result to reduce the problem to computing $\lambda(x)$ in what follows.

\begin{prop}
\label{fundamental_prop}
(\cite[Lemma 4.2]{Ish11}) If $X$ is a variety over $k$ of dimension $d$ and $x$ is a closed point of $X$, then
\begin{equation*}
\lambda(x)=\mld (x;X)-d.
\end{equation*}
\end{prop}

\begin{rem}
\label{relation_smooth}
If $X$ is smooth and $x\in X$ is a closed variety, by Remark \ref{rem_smooth} we have $\lambda_m(x)=0$ for every $m\geq 0$. Hence, by Proposition \ref{fundamental_prop} we have $\mld(x;X)=\dim(X)$. The same conclusion holds if we only assume $x$ is a smooth point of $X$ because in this case we may replace $X$ by a smooth open neighborhood of $x$. Therefore, we only consider singular points in the following chapters.
\end{rem}

\section{Mather minimal log discrepancy of toric varieties}
\label{sec4}
This section is devoted to the computation of the Mather minimal log discrepancy associated to a closed point on a toric variety. We will first do the computation for a torus-invariant point, and then show the computation generalizes to an arbitrary closed point.

The computation depends only on local properties of the toric variety so we assume throughout the section that $X=X(\sigma)$ is an affine toric variety associated to a cone $\sigma$ over an algebraically closed field of characteristic zero.

We write $x_\sigma$ for the torus-invariant point in $X$ (when it exists), and therefore according to Proposition \ref{fundamental_prop}, computing the Mather minimal log discrepancy associated to $x_\sigma$ is equivalent to computing the dimension of $C^m:=\psi_m (\pi^{-1}(x_\sigma))$ for $m$ large enough. More precisely, $\mld (x_\sigma;X)=mn-\dim (C^m)$ when $m\gg 0$. We will decompose $C^m$ into orbits under the $T_m$-action, where $T$ is the torus in $X$, and compute the dimension of $C^m$ by computing the dimension of these orbits instead.

\subsection{Quick review}
$\\$

Let $k$ be an algebraically closed field of characteristic zero. An affine \emph{toric variety} of dimension $n$ is defined using a \emph{lattice} $N\cong\mathbb{Z}^n$ and a cone $\sigma$ in $N_\mathbb{R}:=N\otimes_\mathbb{Z} \mathbb{R}$. A \emph{cone} $\sigma$ is a rational convex cone in $N_\mathbb{R}$ containing no nonzero linear subspace and which is generated by finitely many lattice vectors.

Let $M:=\textnormal{Hom}_\mathbb{Z}(N,
\mathbb{Z})$ be the \emph{dual lattice} and we put $M_\mathbb{R}=M\otimes_\mathbb{Z}\mathbb{R}$. We denote by $\langle \cdot, \cdot \rangle$ the canonical pairing $M\times N\rightarrow \mathbb{Z}$. The affine toric variety associated to the cone $\sigma$ is defined as
\begin{equation*}
X(\sigma):=\Spec\ k[M\cap \sigma^\vee],
\end{equation*}
where $\sigma^\vee$ is the \emph{dual cone} contained in $M_\mathbb{R}$, i.e. $\sigma^\vee=\{u\in M_\mathbb{R} | \langle u,v\rangle\geq 0\textnormal{ for all }v\in \sigma\}$. The semigroup algebra $k[M\cap \sigma^\vee]$ is defined as $\underset{u\in M\cap \sigma^\vee}{\oplus} k\cdot \chi^u$, with $\chi^u\cdot \chi^v=\chi^{u+v}$. Then clearly for some elements $u_1,\ldots,u_s\in M\cap \sigma^\vee$, we have $\chi^{u_1},\ldots,\chi^{u_s}$ generate $k[M\cap \sigma^\vee]$ if and only if $u_1,\ldots,u_s$ generate $M\cap \sigma^\vee$ as a semigroup.

A $k$-valued point $x$ of $X(\sigma)$ corresponds to a homomorphism of $k$-algebras
\begin{equation*}
x^\ast:k[\sigma^\vee\cap M]\longrightarrow k.
\end{equation*}

We put $\sigma^\perp:=\{u\in M_\mathbb{R} | \langle u,v\rangle= 0\textnormal{ for all }v\in \sigma\}$. A \emph{face} of $\sigma$ is a subset of $\sigma$ of the form $\{v\in \sigma| \langle u,v\rangle=0\}$, for some $u\in\sigma^\vee$. The \emph{distinguished point} $x_\tau$ corresponding to a face $\tau$ of $\sigma$ is defined by $x_\tau^\ast(\chi^u)=1$ if $u\in \tau^\perp$, and $x_\tau^\ast(\chi^u)=0$ otherwise.

\begin{rem}
The point $x_\sigma$ exists if $N_\mathbb{R}$ is the linear span of $\sigma$. This point will play a special role in what follows. The computation when $N_\mathbb{R}$ is not spanned by $\sigma$ can be easily reduced to this case, since $X(\sigma)$ will be a product of a torus with a lower-dimensional toric variety that contains a torus-invariant point. So from now on, we assume that $\sigma$ spans $N_\mathbb{R}$.
\end{rem}

The toric variety $X=X(\sigma)$ contains the torus $T=\Spec\ k[M]\cong(k^\ast)^n$ and the group action of $T$ on itself extends to an action on $X$. More precisely, the $T$-action on $X$ is given by $G:T\times X\rightarrow X$, which is equivalent to the following morphism of $k$-algebras:
\begin{equation}
\label{eqn_action}
k[M\cap\sigma^\vee]\longrightarrow k[M]\otimes k[M\cap \sigma^\vee], \chi^u\longmapsto \chi^u\otimes\chi^u.
\end{equation}

It is a general fact that the $T$-orbit $O(\tau)$ that contains $x_\tau$ is of dimension equal to the codimension of $\tau$ in $\sigma$. In particular, the point $x_\sigma$ is the unique torus-invariant point. The toric variety $X$ is the disjoint union of the orbits $O(\tau)$ with $\tau$ varying over all faces of $\sigma$. Therefore, any point of $X$ lies in the same orbit with one of the $x_\tau$'s.

A torus-invariant prime divisor $D$ is the closure of the orbit associated to a one-dimensional face. Let us call this one-dimensional face $\tau$. Then we have
\begin{equation*}
D=V (\tau):=\Spec\ k[M\cap \sigma^\vee\cap \tau^\perp].
\end{equation*}

For more details on toric varieties, we refer the reader to \cite{Ful93}.

\subsection{Characterization of orbits in $C^m$}
\label{section_orbits}

\begin{defn}
\label{defn_Cm}
Recall that for each variety $X$ over $k$ there are canonical morphisms $\psi_m:X_\infty\rightarrow X_m$ and $\pi: X_\infty\rightarrow X$. For every $m\geq 1$, we define $C^m$, a subset of $X_m$, as
\begin{equation*}
C^m:=\psi_m(\pi^{-1}(x_\sigma)).
\end{equation*}
\end{defn}

Let $T$ be the torus in $X=X(\sigma)$. It follows from Remark \ref{jet_rems} (iii) that there is a natural group action of $T_m$ on $X_m$. In this subsection, we approximate $C^m$ by a union of $T_m$-orbits and show that these orbits can be represented by lattice points in the interior of $\sigma$. This characterization builds on the work of Ishii \cite{Ish03} who gave a similar description for the $T_\infty$-orbits in $X_\infty$. We denote by $\mathbb{Z}_{\geq 0}$ the set of nonnegative integers.

Let $\gamma :\Spec\ k[t]/(t^{m+1})\longrightarrow X$ be an $m$-jet inside $C^m$ and let $\delta:\Spec\ k[\![t]\!]\longrightarrow X$ be an arc on $X$ which lifts $\gamma$. Then we have the following commutative diagram:

$$\begin{tikzcd}[column sep=small]
k[M\cap \sigma^\vee]  \arrow{r}{\delta^\ast}  \arrow{rd}{\gamma^\ast}
  & k[\![t]\!]\arrow{d}{} \\
   & k[t]/(t^{m+1}),
\end{tikzcd}$$
where the vertical map is the canonical truncation.

Let $\tau_1,\ldots,\tau_d$ be the one-dimensional faces of $\sigma$ and $D_i:=V(\tau_i)$ be the corresponding torus-invariant prime divisors of $X$. We assume that $\delta$ is not in the arc space of any $D_i$. Equivalently, $\delta^\ast(\chi ^u)\neq 0$ for every $u\in M\cap\sigma^\vee$. Thus the order in $t$ of $\delta^\ast(\chi ^u)\in k[\![t]\!]$ is well-defined. We put $\ord_\delta(u):=\ord_t(\delta^\ast(\chi^u))$ for each $u\in M\cap \sigma^\vee$. Let $S_m$ be the set $\{0,1,\ldots,m,\infty\}$ and $Tr_m:\mathbb{Z}_{\geq 0} \rightarrow S_m$ be the obvious truncation map that takes any number larger than $m$ to $\infty$. We define $\ord_\gamma$ to be the composition of $\ord_\delta$ with $Tr_m$. Then we get the following commutative diagram:

$$\begin{tikzcd}[column sep=large]
M\cap \sigma^\vee  \arrow{r}{\ord_\delta}  \arrow{rd}{\ord_\gamma}
  & \mathbb{Z}_{\geq 0}\arrow{d}{Tr_m} \\
    & S_m.
\end{tikzcd}$$

Note that for each $u\in M\cap \sigma^\vee$, the value of $\ord_\gamma(u)$ only depends on $\gamma^\ast(\chi^u)$. Thus $\ord_\gamma$ is independent of choice of $\delta$ and we call it the order map of $\gamma$.

Since $\ord_\delta$ is an additive map that takes lattice points in the cone $\sigma^\vee$ to nonnegative integers, it corresponds uniquely to a lattice point in $\sigma$. Now we give a first description of some of the $T_m$-orbits of $C^m$.

\begin{lem}
\label{lem_orbits}
With the above notation, $\psi_m(X_\infty\backslash \underset{i}{\cup} (D_i)_\infty)$ is preserved by the $T_m$-action and its orbits are in one-to-one correspondence with the maps $M\cap \sigma^\vee\rightarrow S_m$ that can be lifted to additive maps $M\cap \sigma^\vee\rightarrow\mathbb{Z}_{\geq 0}$. The corresponding map is exactly the order map of any element in the orbit.

Moreover, $\psi_m(\pi^{-1}(x_\sigma)\backslash \underset{i}{\cup} (D_i)_\infty)$ is also preserved by the $T_m$-action and its orbits are in one-to-one correspondence with the maps $M\cap \sigma^\vee\rightarrow S_m$ that can be lifted to additive maps $M\cap \sigma^\vee \rightarrow \mathbb{Z}_{\geq 0}$, such that the inverse image of $\{0\}$ is $\{0\}$.
\end{lem}

\begin{proof}
First let us describe the $T_m$-action on $X_m$ and the $T_\infty$-action on $X_\infty$. Let
\begin{equation*}
g:\Spec\ k[t]/(t^{m+1})\rightarrow \Spec\ k[M]
\end{equation*}
be a point of $T_m$ and
\begin{equation*}
\gamma:\Spec\ k[t]/(t^{m+1})\rightarrow \Spec\ k[M\cap \sigma^\vee]
\end{equation*}
be a point of $X_m$. Then $g\cdot \gamma$ is a morphism $\Spec\ k[t]/(t^{m+1})\rightarrow \Spec\ k[M\cap \sigma^\vee]$ that is equal to $G\circ (g,\gamma)$. By equation (\ref{eqn_action}), for each $u\in M\cap \sigma$ we have
\begin{equation*}
(g\cdot \gamma)^\ast (\chi^u)=g^\ast(\chi^u)\cdot \gamma^\ast(\chi^u).
\end{equation*}
Similarly, if $\delta:\Spec\ k[\![t]\!]\rightarrow k[M]$ is a point in $T_\infty$ and $\alpha:\Spec\ k[\![t]\!]\rightarrow k[M\cap \sigma^\vee]$ is a point of $X_\infty$, then for each $u\in M\cap \sigma^\vee$ we have
\begin{equation*}
(\delta\cdot\alpha)^\ast(\chi^u)=\delta^\ast(\chi^u)\cdot \alpha^\ast(\chi^u).
\end{equation*}
Note that both $g^\ast(\chi^u)$ and $\delta^\ast(\chi^u)$ above are units since $\chi^u$ has an inverse $\chi^{-u}$ in $k[M]$.

Now let $\gamma$ be an $m$-jet in $\psi_m(X_\infty\backslash \underset{i}{\cup} (D_i)_\infty)$ with a lifting $\delta$ in $X_\infty\backslash \underset{i}{\cup} (D_i)_\infty$. For each $\alpha\in T_m$, there is a lifting $\xi\in T_\infty$ of $\alpha$ by smoothness of $T$. Since $\xi^\ast(\chi^u)$ is a unit for each $u\in M$, $(\xi\cdot \delta)^\ast(\chi^u)\neq 0$ for each $u\in M\cap \sigma^\vee$. Hence $\xi\cdot \delta$ is also in $X_\infty\backslash \underset{i}{\cup} (D_i)_\infty$. Therefore, $\alpha\cdot \gamma=\psi_m(\xi\cdot \delta)$ is in $\psi_m(X_\infty\backslash \underset{i}{\cup} (D_i)_\infty)$. This shows that $\psi_m(X_\infty\backslash \underset{i}{\cup} (D_i)_\infty)$ is preserved by the $T_m$-action.

For $\psi_m(\pi^{-1}(x_\sigma)\backslash\underset{i}{\cup} (D_i)_\infty)$, one applies the same argument and observes that an arc $\delta$ lies above the torus-invariant point $x_\sigma$ if and only if $\delta^\ast(\chi^u)$ has positive order whenever $u\neq 0$, which is equivalent to $\ord_\delta^{-1}(0)=\{0\}$. Since $\xi^\ast(\chi^u)$ is a unit for each $u\in M$ and $\xi\in T_\infty$, $\xi\cdot \delta$ also lies above $0$. This shows that $\psi_m(\pi^{-1}(x_\sigma)\backslash \underset{i}{\cup} (D_i)_\infty)$ is also preserved by the $T_m$-action.

Pick two $m$-jets $\alpha\in T_m$ and $\gamma\in \psi_m(X_\infty\backslash \underset{i}{\cup} (D_i)_\infty)$. The morphism $\alpha^\ast$ takes any $\chi^u$, with $u\in M$, to a unit. Therefore, multiplying $\gamma$ by $\alpha$ does not change the order map $\ord_\gamma$. In other words, $\ord_{\gamma}=\ord_{ \alpha\cdot\gamma}$. This shows that the order map is the same for all points in a $T_m$-orbit.

Now we show that two $m$-jets in $\psi_m(X_\infty\backslash \underset{i}{\cup} (D_i)_\infty)$ with the same order map are in the same $T_m$-orbit. Let $\gamma$ be in $\psi_m(X_\infty\backslash \underset{i}{\cup} (D_i)_\infty)$ and $\phi$ be its order map. We define the special $m$-jet $\gamma_\phi$ whose associated morphism is
\begin{equation*}
\gamma_\phi^\ast:k[M\cap \sigma^\vee]\longrightarrow k[t]/(t^{m+1}),\ \gamma_\phi^\ast(\chi^u)= t^{\phi(u)},
\end{equation*}
with the convention that $t^\infty=0$. If we write $\phi(a)+\phi(b)=\infty$ whenever the sum is $\geq m+1$, then we have $\phi(a)+\phi(b)=\phi(a+b)$ for any $a,b\in M\cap\sigma^\vee$. Therefore, $\gamma_\phi^\ast$ is a homomorphism of $k$-algebras.

Let $\delta\in X_\infty\backslash \underset{i}{\cup} (D_i)_\infty$ be a lifting of $\gamma$ and $\psi$ be the order map of $\delta$. Then we may define an arc $\delta_\psi$ such that
\begin{equation*}
\delta_\psi^\ast:k[M\cap \sigma^\vee]\longrightarrow k[\![t]\!],\ \delta_\psi^\ast(\chi^u)=t^{\psi(u)}.
\end{equation*}
Obviously, $\delta_\psi$ lifts $\gamma_\phi$, and it has the same order map as $\delta$. Hence we have a morphism of $k$-algebras as follows:
\begin{equation*}
\alpha^\ast:k[M\cap \sigma^\vee]\longrightarrow k[\![t]\!],\ \alpha^\ast(\chi^u)=\delta^\ast(\chi^u)/ \delta_\psi^\ast(\chi^u).
\end{equation*}
$\alpha^\ast$ extends to the entire $k[M]$ since $M\cap \sigma^\vee$ spans $M$ and since $\alpha^\ast(\chi^u)$ is a unit for each $u\in M\cap \sigma^\vee $. Hence $\alpha\in T_\infty$ and clearly we have $\alpha\cdot \delta_\psi=\delta$, and therefore, $\psi_m(\alpha)\cdot \gamma_\phi=\gamma$. This shows that $\gamma$ is in the same $T_m$-orbit as the special $m$-jet $\gamma_\phi$, and so is any other $m$-jet with the same order map.

Finally, we show that each map $\phi:M\cap \sigma^\vee\rightarrow S_m$ that can be lifted to an additive map $\psi: M\cap \sigma^\vee\rightarrow \mathbb{Z}_{\geq 0}$ is the order map of some $m$-jet in $\psi_m(X_\infty\backslash \underset{i}{\cup} (D_i)_\infty)$. Define the special $m$-jet $\gamma_\phi$ and the arc $\delta_\psi$ that lifts $\gamma_\phi$ in the same way as above. Then we have $\delta_\psi \in X_\infty\backslash\underset{i}{\cup} (D_i)_\infty$ because $\delta_\psi^\ast(\chi^u)$ has finite order for each $u\in M\cap \sigma^\vee$. Therefore, $\gamma_\phi$ is an $m$-jet in $\psi_m(X_\infty\backslash\underset{i}{\cup} (D_i)_\infty)$ and clearly $\ord_{\gamma_\phi}=\phi$. Hence we have produced a $T_m$-orbit whose corresponding order map is equal to the map $\phi$ that we started with.
\end{proof}

As mentioned above, an additive map $M\cap \sigma^\vee\rightarrow \mathbb{Z}_{\geq 0}$ corresponds uniquely to a lattice point in $\sigma$. Denote by $\varphi_a$ the additive map corresponding to the lattice point $a$ and $\bar{\varphi}_a$ the composition of $\varphi_a$ with the truncation map $Tr_m$. Then clearly every order map in Lemma \ref{lem_orbits} is equal to $\bar{\varphi}_a$ for some $a\in \sigma\cap N$. In particular, the order map takes only $0$ to $0$ if the lattice point $a$ is contained in $\Int(\sigma)$, the interior of $\sigma$. However, there could be more than one such $a$. To understand the additive maps $M\cap \sigma^\vee\rightarrow \mathbb{Z}_{\geq 0}$ better, we first study the semigroup $M\cap \sigma^\vee$ and show that there is a unique minimal set of generators.

\begin{defn}
An element $u\in M\cap \sigma^\vee$ is called \emph{irreducible} if it cannot be written as the sum of two nonzero elements of $M\cap \sigma^\vee$.
\end{defn}

\begin{lem}
The semigroup $M\cap\sigma^\vee$ has a unique minimal set of generators consisting of all the irreducible elements.
\end{lem}

\begin{proof}
First, since $\sigma^\vee$ is a convex polyhedral cone, $M\cap \sigma^\vee$ is finitely generated. Therefore, there exists a minimal set of generators.

Second, we show that any element of $M\cap\sigma^\vee$ can be generated by irreducible elements. Pick an element $v\in \Int(\sigma)\cap N$. Then $\langle u,v\rangle$ is a positive integer for any $u\in M\cap \sigma^\vee$. We claim that for each $u\in M\cap \sigma^\vee$, $u$ can be written as the sum of at most $\langle u,v\rangle$ irreducible elements. If $\langle u,v\rangle=1$, then $u$ must irreducible. Otherwise, there are nonzero elements $u_1,u_2\in M\cap \sigma^\vee$ such that $u=u_1+u_2$. But $\langle u_1,v\rangle$ and $\langle u_2,v\rangle$ are both positive integers since $v\in \Int(\sigma)\cap N$. This is not possible as they add up to $\langle u,v\rangle=1$. Inductively, suppose our claim holds for all $u$ such that $\langle u,v\rangle\leq p$. Pick $u\in M\cap \sigma^\vee$ such that $\langle u,v\rangle=p+1$. If $u$ is irreducible, then we are done. Otherwise, there are nonzero elements $u_1,u_2\in M\cap \sigma^\vee$ such that $u=u_1+u_2$. Both $\langle u_1,v\rangle$ and $\langle u_2,v\rangle$ are $\leq p$. By assumption, $u_1$ and $u_2$ can be written as the sum of at most $\langle u_1,v\rangle$ and $\langle u_2,v\rangle$ irreducible elements respectively. Therefore, $u$ can be written as the sum of at most $\langle u_1,v\rangle+\langle u_2,v\rangle=p+1$ irreducible elements.

Finally, note that any set of generators must contain all irreducible elements by definition. We conclude that the set of irreducible elements form the unique minimal set of generators for $M\cap \sigma^\vee$.
\end{proof}

\begin{rem}
If $u_1,\ldots,u_s$ form the unique minimal set of generators of $M\cap \sigma^\vee$, then $\chi^{u_1},\ldots,\chi^{u_s}$ also form the unique minimal set of monomial generators of $k[M\cap \sigma^\vee]$.
\end{rem}

The following lemma makes a connection between the set of order maps and the set of lattice points.

\begin{lem}
\label{lem_compactness}
Fix an integer $m\geq 1$ and let $\chi^{u_1},\chi^{u_2},\ldots,\chi^{u_s}$ be the minimal set of monomial generators of $k[M\cap\sigma^\vee]$. For each integer $c\geq 0$ we define
\begin{equation*}
P_c:=\Big\{ a\in \sigma\cap N \Big |\textnormal{the set}\ \{u_i|\varphi_a(u_i)\leq m+c \}\ \textnormal{spans}\ M_{\mathbb{R}}\Big\}.
\end{equation*}
Then the following hold:\\
(1) For any two different $a,b\in P_0$, $\bar{\varphi}_a\neq \bar{\varphi}_b$.\\
(2) There exists some $c_0\in \mathbb{Z}^+$ such that for any $a\in \sigma\cap N$ one can find $b\in P_{c_0}$ with $\bar{\varphi}_a= \bar{\varphi}_b$.
\end{lem}

\begin{proof}
First let's assume we have $a,b\in P_0$ and $\bar{\varphi}_a= \bar{\varphi}_b$. Define
\begin{equation*}
\Gamma_0:=\{u_i|\varphi_a(u_i)\leq m\}.
\end{equation*}

Since $\bar{\varphi}_a=\bar{\varphi}_b$, we deduce that $\varphi_a$ and $\varphi_b$ take the same values on $\Gamma_0$. By definition of $P_0$, $\Gamma_0$ spans $M_\mathbb{R}$. Thus we conclude that $\varphi_a=\varphi_b$, which implies that $a=b$.\\

For (2), we choose a positive integer $c_0$ large enough such that for any subset $S\subset \{u_1,u_2,\ldots,u_s\}$ that does not span $M_\mathbb{R}$, there is some $v\in N$ satisfying
\begin{equation}
\label{eqn_2.5}
\varphi_v(u_i)=0,\textnormal{ for all } u_i\in S,\ \textnormal{and }1\leq \max_{u_i\notin S}\{\varphi_v(u_i)\}\leq c_0.
\end{equation}

Such a number $c_0$ exists because there are only finitely many subsets of $\{1,2,\ldots,s\}$. For each point $b\in \sigma\cap N$ we put $S_b:=\big\{u\in \{u_1,\ldots,u_s\}|\varphi_b(u)\leq m+c_0\big\}$. If there is some $b$ such that $\bar\varphi_a=\bar\varphi_b$ and such that $S_b$ spans $M_\mathbb{R}$, then we are done.

Now suppose there is no such $b$. We pick a point $b$ such that $\bar\varphi_a=\bar\varphi_b$ and such that $S_b$ is maximal. By relabeling we may write $S_b=\{u_1,\ldots,u_l\}$ for some integer $l<s$. By assumption $S_b$ does not span $M_\mathbb{R}$, so we can fine $v\in N$ that satisfies (\ref{eqn_2.5}) with $S$ replaced by $S_b$. Clearly there is some positive integer $k$ such that
\begin{equation*}
\varphi_{b-kv}(u_i) >m,\textnormal{ for all } i>l,
\end{equation*}
\begin{equation*}
\varphi_{b-kv}(u_{i_0}) \leq m+c_0,\textnormal{ for some } i_0>l.
\end{equation*}
Notice that $\bar{\varphi}_{b-kv}=\bar{\varphi}_b=\bar\varphi_a$, and hence $b-kv\in \sigma\cap N$. But clearly we have
\begin{equation*}
S_b\subsetneqq\{u_1,\ldots,u_l,u_{i_0}\}\subset S_{b-kv}.
\end{equation*}
This contradicts the maximality of $S_b$. So we conclude that there must be some $b\in P_{c_0}$ such that $\bar\varphi_a=\bar\varphi_b$.
\end{proof}

\begin{rem}
We have proved that for each $a\in \sigma\cap N$, the map $\bar{\varphi}_a$ corresponds to a $T_m$-orbit in $\psi_m(X_\infty\backslash \cup_i (D_i)_\infty)$. We denote this orbit by $T_{m,a}$.
\end{rem}

\begin{rem}
For each $a\in \Int(\sigma)\cap N$, $\varphi_a$ is an additive map $M\cap \sigma^\vee\rightarrow\mathbb{Z}_{\geq 0}$ such that $\varphi_a^{-1}(0)=\{0\}$. According to Lemma \ref{lem_orbits} the corresponding orbits $T_{m,a}$ are all the $T_m$-orbits contained in $\psi_m(\pi^{-1}(x_\sigma)\backslash \cup_i (D_i)_\infty)$.
\end{rem}

\begin{cor}
\label{cor_finiteness}
The sets $\psi_m(X_\infty\backslash \underset{i}{\cup} (D_i)_\infty)$ and $\psi_m(\pi^{-1}(x_\sigma)\backslash \underset{i}{\cup} (D_i)_\infty)$ contain only finitely many $T_m$-orbits.
\end{cor}

\begin{proof}
According to Lemma \ref{lem_orbits}, we just need to show there are finitely many order maps $\bar\varphi_a$ for $a\in \sigma\cap N$. By Lemma \ref{lem_compactness}, there is a positive integer $c_0$ such that every order map is equal to $\bar\varphi_a$ for some $a\in P_{c_0}$. Therefore, it suffices to show that $P_{c_0}$ is compact.

For any $u_{j_1},\ldots,u_{j_n}\subset \{u_1,\ldots,u_s\}$ that span $M_\mathbb{R}$, we define
\begin{equation*}
K_{j_1,j_2,\ldots,j_n}:= \{ a\in \sigma\cap N|\varphi_a(u_{j_i})\leq m+c_0\textnormal{ for }1\leq i\leq n\}.
\end{equation*}
Then $P_{c_0}$ is the union of all $K_{j_1,\ldots,j_n}$ as $(j_i)_{1\leq i\leq n}$ varies such that $u_{j_1},\ldots,u_{j_n}$ span $M_\mathbb{R}$. Since this is a finite union, it suffices to show that each $K_{j_1,\ldots,j_n}$ is compact.

By relabeling let us assume that $j_i=i$ for $1\leq i\leq n$. Let $v_1,\ldots,v_l$ be a minimal set of generators of $\sigma\cap N$. Since $u_1,\ldots,u_n$ span $M_\mathbb{R}$, for each $v_i$ there exists some $u_j$ with $1\leq j\leq n$ such that $\langle v_i,u_j\rangle$ is a positive integer. Therefore,
\begin{equation*}
K_{1,2,\ldots,n}\subset \{a\in \sigma\cap N|a=\sum_{i=1}^l c_i v_i,\textnormal{ with }0\leq c_i\leq m+c_0\textnormal{ for each } i\}.
\end{equation*}
This shows that $K_{1,2,..,n}$ is compact.
\end{proof}

\begin{rem}
The structure of the jet schemes of toric varieties is in general very hard to describe unlike the case of arc spaces. One can find a description of jet schemes of toric surfaces in \cite{Mou11}. Instead of the entire jet schemes, we only describe the structure of images of the arc space in the $m^{\textnormal{th}}$ jet scheme.
\end{rem}

\subsection{Main results}
\label{section_dimension}
$\\$
In this subsection we compute the dimension of the orbit $T_{m,a}$ by computing the dimension of the corresponding stabilizer. Denote by $H_{m,a}$ the stabilizer of any element of $T_{m,a}$ under the $T_m$-action. We start with the following lemma.

\begin{lem}
\label{lem_finiteness}
Let $u_1,\ldots,u_n\in M$ be elements that generate $M_{\mathbb{R}}$ over $\mathbb{R}$. For every $a_{i,j}\in k$ with $1\leq i\leq n$ and $0\leq j\leq m$ such that $a_{i,0}\neq 0$ for all $i$, the set of elements $\alpha\in T_m$ such that
\begin{equation}
\label{eqn_4}
\alpha^\ast(\chi^{u_i})=\sum_{j=0}^m a_{i,j}t^j\textnormal{ for }1\leq i\leq n
\end{equation}
is nonempty and finite.
\end{lem}

\begin{proof}
Consider the subgroup $M'$ of $M$ generated by $u_1,\ldots,u_n$ and the corresponding torus $T'=\Spec\ k[M']$. Note that we have an induced morphism $f\colon T\to T'$. It is well-known that in characteristic $0$, this map is finite and \'{e}tale. This follows, for example, by choosing a basis $w_1,\ldots,w_n$ of $M$ such that $d_1w_1,\ldots,d_nw_n$ is a basis of $M'$, for some positive integers $d_1,\ldots,d_n$. In this case, it follows from Lemma \ref{lem_etale} that
\begin{equation*}
T_m\simeq T'_m\times_{T'}T.
\end{equation*}
In particular, the induced morphism $T'_m\to T_m$ is finite and \'{e}tale and its fibers are non-empty and finite. Since it is clear that there is a unique $\beta\in T'_m$ such that $\beta^*(\chi^{u_i})=\sum_{j=0}^ma_{i,,j}t^j$ for all $i$, we deduce the assertion in the lemma.
\end{proof}

\begin{defn}
\label{defn_ph}
For each $a\in \Int(\sigma)\cap N$, we define
\begin{equation}
\label{eqn_ph}
\Phi(a):= \min \big\{\sum_{i=1}^n \langle a,u_{i} \rangle | u_{1},\ldots,u_{n}\ \textnormal{span}\  M_\mathbb{R},\textnormal{ with }u_i\in M\cap \sigma^\vee\textnormal{ for each } i\big\},
\end{equation}
where the minimum is run over all linearly independent sets of vectors $\{u_1,\ldots,u_n\}$ in $M\cap \sigma^\vee$.
\end{defn}

\begin{rem}
\label{rem_ph}
Clearly if the minimum in (\ref{eqn_ph}) is attained at some elements $u_1,\ldots,u_n$, each $u_i$ must be irreducible. We show in the following one way to find elements $u_1,\ldots,u_n$ at which the above minimum is achieved.
\end{rem}

Fix $a\in \Int(\sigma)\cap N$. Let $u_1,\ldots, u_s$ be the minimal set of generators of the semigroup $M\cap \sigma^\vee$ and let $S_0:=\{u_1,\ldots,u_s\}$. We first choose $u_{j_1}\in S_0$ such that $\varphi_a(u_{j_1})=\langle a, u_{j_1}\rangle$ is minimal and define $S_1:=S_0\backslash \textnormal{Span}(u_{j_1})$. Recursively, for each $1\leq i\leq n-1$, assuming $u_{j_1},\ldots, u_{j_i}$ are chosen and $S_i=S_0\backslash \textnormal{Span}(u_{j_1},\ldots,u_{j_{i}})$, we choose $u_{j_{i+1}} \in S_i$ such that $\varphi_a(u_{j_{i+1}})=\langle a, u_{j_{i+1}}\rangle$ is minimal and define $S_{i+1}:=S_0\backslash \textnormal{Span}(u_{j_1},\ldots,u_{j_{i+1}})$. Once $u_{j_1},\ldots, u_{j_n}$ are all chosen, it is clear that they span $M_\mathbb{R}$.

\begin{lem}
\label{lem_ph}
For each $a\in \Int(\sigma)\cap N$ and $u_{j_1},\ldots, u_{j_n}$ chosen as above, we have
\begin{equation*}
\sum_{k=1}^n \langle a,u_{j_k} \rangle=\Phi(a).
\end{equation*}
\end{lem}

\begin{proof}
By Remark \ref{rem_ph} we can find $i_1,\ldots, i_n$ such that $u_{i_1},\ldots, u_{i_n}$ span $M_\mathbb{R}$ and they compute $\Phi(a)$. If the set $\{u_{i_1},\ldots, u_{i_n}\}$ is equal to $\{u_{j_1},\ldots, u_{j_n}\}$, the claim in the lemma follows immediately. Hence we assume that by relabeling, there exists some $k$, with $1\leq k\leq n$, such that $i_1=j_1,\ldots, i_{k-1}=j_{k-1}$ and $i_k\neq j_k$. If $k=n$, we have $\langle a,u_{j_n}\rangle \leq \langle a,u_{i_n}\rangle$ by the choice of $u_{j_n}$. Hence
\begin{equation*}
\sum_{k=1}^n \langle a,u_{j_k} \rangle\leq \sum_{k=1}^n \langle a,u_{i_k}\rangle.
\end{equation*}
This proves the claim in the lemma.

Now suppose the conclusion holds when $k>k_0$ for some $k_0<n$, and we consider the case when $k=k_0$. We claim there exists some $l$, with $k_0\leq l\leq n$, such that
\begin{equation*}
u_{j_{k_0}}\not\in \textnormal{Span}(u_{i_1},\ldots, \hat{u}_{i_l},\ldots, u_{i_n}).
\end{equation*}
Otherwise, we have
\begin{eqnarray*}
u_{j_{k_0}}&\in& \bigcap_{l=k_0}^n \textnormal{Span}(u_{i_1},\ldots, \hat{u}_{i_l},\ldots, u_{i_n})\\
&=& \textnormal{Span}(u_{i_1},\ldots, u_{i_{k_0-1}})\\
&=& \textnormal{Span}(u_{j_1},\ldots, u_{j_{k_0-1}}).
\end{eqnarray*}
But this contradicts the fact that $u_{j_1},\ldots, u_{j_n}$ span $M_\mathbb{R}$.

The above claim implies that if we replace $u_{i_l}$ by $u_{j_{k_0}}$, $u_{i_1},\ldots, u_{i_n}$ still span $M_\mathbb{R}$. It also shows that
\begin{equation*}
u_{i_l}\not\in \textnormal{Span}(u_{j_1},\ldots, u_{j_{k_0-1}}),
\end{equation*}
and hence by the choice of $u_{j_{k_0}}$, we have $\langle a, u_{j_{k_0}} \rangle \leq \langle a, u_{i_l}\rangle$. We conclude that if we replace $u_{i_l}$ by $u_{j_{k_0}}$, the question is reduced to the case when $k\geq k_0+1$, and we are done by induction.
\end{proof}

\begin{thm}
\label{thm_stablizer}
Fix a lattice point $a\in \Int(\sigma)$. Let $\chi^{u_1},\chi^{u_2},\ldots,\chi^{u_s}$ be the minimal set of monomial generators of $k[M\cap\sigma^\vee]$. If $H_{m,a}$ is the stabilizer of any element of $T_{m,a}$ under the $T_m$-action, then the following hold:

(1) We have $\dim(H_{m,a})= \Phi(a)$ for
\begin{equation}
\label{eqn_big}
m\geq \max\{\max_{1\leq i\leq n} \langle a,u_{j_i}\rangle\},
\end{equation}
where the maximum is taken over all possible choices of $n$ vectors $u_{j_1},...,u_{j_n}$ among $u_1,u_2,...,u_s$ that span $M_\mathbb{R}$, and such that the minimum in (\ref{eqn_ph}) is attained.

(2) If $m$ does not satisfy the inequality (\ref{eqn_big}), then we have either $\dim(H_{m,a})= \Phi(a)$ or $m\leq \dim(H_{m,a})\leq \Phi(a)$.
\end{thm}

\begin{proof}
For simplicity we write $\varphi_m$ for $\min\{m,\bar{\varphi}_a\}$, and $\varphi^m$ for $\min\{m+1,\bar{\varphi}_a\}$. Let $H_{m,a}$ be the stabilizer of the special jet $\gamma_{\bar{\varphi}_a}$ defined in Lemma \ref{lem_orbits}. Then an $m$-jet $\alpha\in T_m$ is contained in $H_{m,a}$ if and only if
\begin{equation}
\label{eqn_6}
\alpha^\ast(\chi^{u_i})\cdot t^{\varphi^m(u_i)}=t^{\varphi^m(u_i)}\textnormal{ in } k[t]/(t^{m+1}) \textnormal{ for }1\leq i\leq s.
\end{equation}
This is clearly equivalent to
\begin{equation}
\label{eqn_7}
\begin{split}
\alpha^\ast (\chi^{u_i})=1+\sum_{j=m+1-\varphi^m(u_i)}^m a_{i,j}t^j,\ \textnormal{if }\varphi^m(u_i)\leq m,\\
\alpha^\ast (\chi^{u_i})=\sum_{j=0}^m a_{i,j}t^j,\ \textnormal{if }\varphi^m(u_i)=m+1,
\end{split}
\end{equation}
for each $i$ and for some $a_{i,j}\in k$, with the condition that $a_{i,0}\neq 0$ when $\varphi^m(u_i)=m+1$.

Choose any $n$ vectors from $\{u_1,\ldots,u_s\}$ that span $M_\mathbb{R}$. By relabeling, let us assume they are $u_1,\ldots,u_n$.
 We define $A:=\mathbb{A}^{\sum_{i=1}^n \varphi^m(u_i)}$ and the map
\begin{equation*}
\pi:H_{m,a}\longrightarrow A,\ \pi(\alpha)=(a_{i,j})_{1\leq i\leq n,m+1-\varphi^m(u_i)\leq j\leq m}.
\end{equation*}
Then Lemma \ref{lem_finiteness} implies that $\pi$ has finite fibers. Therefore, we have
\begin{equation*}
\dim(H_{m,a})\leq\dim(A)= \sum_{i=1}^n\varphi^m(u_i).
\end{equation*}
By letting $u_1,\ldots,u_n$ vary so that they span $M_\mathbb{R}$, we conclude that
\begin{equation*}
\dim(H_{m,a})\leq\min\{\sum_{i=1}^n \varphi^m(u_{j_i})|u_{j_1},\ldots,u_{j_n}\textnormal{ span }M_\mathbb{R}\}\leq \Phi(a).
\end{equation*}\\

In what follows, we assume that after relabeling, $u_1,\ldots, u_n$ are chosen as in Lemma \ref{lem_ph}. We claim that
\begin{equation}
\label{eqn_8}
\dim(H_{m,a})\geq \sum_{i=1}^n \varphi_m(u_i).
\end{equation}
Consider the subgroup $M'$ of $M$ generated by $u_1,\ldots,u_n$ and the corresponding torus $T'=\Spec\ k[M']$. By the proof of Lemma \ref{lem_finiteness}, the commutative diagram
$$\begin{CD}
T_m @> >> T'_m\\
@V\pi_m^T VV @V\pi_m^{T'} VV\\
T @> >> T'.
\end{CD}$$
is Cartesian. Hence, for each $\alpha'\in T_m'$ that lies over $(1,\ldots, 1)\in T'$, there is a unique $\alpha\in T_m$ lying over $(1,\ldots, 1)\in T$ such that $\alpha$ is mapped to $\alpha'$. We claim that for each $\alpha'\in T_m'$ lying over $(1,\ldots, 1)$ that satisfies (\ref{eqn_7}) for $1\leq i\leq n$, the corresponding $\alpha\in T_m$ is an element in $H_{m,a}$.

To prove this, we just need to show that $\alpha$ satisfies conditions (\ref{eqn_7}) for $1\leq i\leq s$. Since $\alpha$ maps to $\alpha'$, it automatically satisfies (\ref{eqn_7}) for $1\leq i\leq n$. Now pick an integer $z$ such that $n+1\leq z\leq s$. Then there exist integers $l>0$, $d_i$ and $q\leq n$ such that
\begin{equation}
\label{eqn_7.5}
lu_z=\sum_{i=1}^q d_i u_i,
\end{equation}
where $d_q\neq 0$. By applying $\alpha^\ast$ on both sides, we get
\begin{equation*}
\alpha^\ast(\chi^{u_z})^l=\prod_{i=1}^q \alpha^\ast(\chi^{u_i})^{d_i}.
\end{equation*}
By using (\ref{eqn_7}) for $1\leq i\leq n$, we see that the $t$-order of $\prod_{i=1}^q \alpha^\ast(\chi^{u_i})^{d_i}-1$, hence also that of $\alpha^\ast(\chi^{u_z})^l-1$, is at least $\min_{1\leq i\leq q} \{m+1-\varphi_m(u_i)\}$. Since $\alpha$ lies over $(1,\ldots, 1)$, this implies that the $t$-order of $\alpha^\ast(\chi^{u_z})-1$ is at least $\min_{1\leq i\leq q} \{m+1-\varphi_m(u_i)\}$. On the other hand, equation (\ref{eqn_7.5}) implies that
\begin{equation*}
u_z\in \{u_1,\ldots,u_s\}\backslash \textnormal{Span}(u_1,\ldots, u_k),
\end{equation*}
for each $k$ with $1\leq k\leq q-1$. By the construction of $u_1,\ldots, u_n$, we have $\varphi_a(u_z)\geq \varphi_a(u_k)$ for each $1\leq k\leq q$. Hence, we have
\begin{equation*}
m+1-\varphi^m(u_z)\leq \min_{1\leq i\leq q} \{m+1-\varphi^m(u_i)\}.
\end{equation*}
So the $t$-order of $\alpha^\ast(\chi^{u_z})-1$ is $\geq m+1-\varphi^m(u_z)$. This, however, implies condition (\ref{eqn_7}) for $i=z$. Since $z$ is arbitrary, $\alpha$ satisfies conditions (\ref{eqn_7}) for $1\leq i\leq s$, and hence $\alpha\in H_{m,a}$.

Define the affine space $A$ and the map $\pi:H_{m,a}\rightarrow A$ as above with respect to $u_1,\ldots, u_n$. Let $Y\subset A$ be the subspace defined by $a_{i,0}=1$ for $1\leq i\leq n$ such that $\varphi^m(u_i)=m+1$. Then the above discussion shows that $Y$ is contained in the image of $\pi$. We conclude that
\begin{equation*}
\dim(H_{m,a})\geq \dim(Y)=\sum_{i=1}^n \varphi_m(u_i).
\end{equation*}

According to Lemma \ref{lem_ph}, the minimum in (\ref{eqn_ph}) is achieved by $u_1,\ldots, u_n$. Hence, the condition (\ref{eqn_big}) guarantees that $m\geq \varphi_a(u_i)$ for each $1\leq i\leq n$. Under this condition, we have
\begin{equation*}
\dim(H_{m,a})\geq \sum_{i=1}^n \varphi_m(u_i)= \sum_{i=1}^n \varphi_a(u_i)=\Phi(a).
\end{equation*}
This completes the proof of (1).

For (2), we consider two cases. If $m\geq \max_{1\leq i\leq n}\varphi_a(u_i)$, then (\ref{eqn_8}) implies that $\dim(H_{m,a})\geq \Phi(a)$ as in (1). If there is some $i$, with $1\leq i\leq n$, such that $m<\varphi_a(u_i)$, then $\varphi_m(u_i)=m$. So (\ref{eqn_8}) implies that $\dim(H_{m,a})\geq \varphi_m(u_i)=m$. Since we have proved that $\dim(H_{m,a})$ is always $\leq \Phi(a)$, the conclusions in (2) follow.
\end{proof}

\begin{cor}
\label{cor_orbits}
With the same assumptions as in Theorem \ref{thm_stablizer} and for all $m\geq 0$, the dimension of the orbit $T_{m,a}$ satisfies one of the following:
\begin{equation}
\label{eqn_8.1}
\dim(T_{m,a})=(m+1)n- \Phi(a),\textnormal{ or}
\end{equation}
\begin{equation}
\label{eqn_8.2}
(m+1)n- \Phi(a)\leq \dim(T_{m,a})\leq (m+1)n-m.
\end{equation}
\end{cor}

\begin{proof}
Observe that $T$ is smooth of dimension $n$. Hence by Corollary \ref{jet_smooth}, $\dim(T_m)=n(m+1)$. The conclusions follow immediately from Theorem \ref{thm_stablizer}.
\end{proof}

Now we can prove our main result. Recall that the Mather minimal log discrepancy can be computed in terms of the invariant $\lambda$ defined in Definition \ref{defn_lambda}, via Property \ref{fundamental_prop}. According to Lemma \ref{lem_lambda}, this in turn can be computed from the dimension of $C^m$ (defined in Definition \ref{defn_Cm}), when $m$ is large enough. We have seen that $C^m$ can be approximated by a union of explicit $T_m$-orbits. Thus computing the dimension of $C^m$ boils down to computing the dimension of these $T_m$-orbits.

\begin{thm}
\label{thm_toric}
For $m$ large enough we have
\begin{equation}
\dim(C^m)=n(m+1)-\min_{a\in \Int(\sigma)\cap N} \Phi(a),
\end{equation}
where $\Phi$ is defined in Definition \ref{defn_ph}.
\end{thm}

\begin{proof}
First of all, note that $T_{m,a}$ lies over the torus-fixed point $x_\sigma$ if and only if $a$ is in the interior of the cone $\sigma$ (see Lemma \ref{lem_orbits}). Therefore $C^m$ is the union of finitely many $T_m$-orbits $T_{m,a}$ (by Lemma \ref{lem_orbits} and Corollary \ref{cor_finiteness}), for $a$ in the interior of $\sigma$, and of the orbits contained in the image of the $(D_i)_\infty$. But $\dim (\psi_m((D_i)_\infty)\leq (n-1)(m+1)$ by Lemma \ref{lem_4.3}. When $m$ is large enough, the dimension of these orbits contained in the image of the $(D_i)_\infty$ is smaller than $mn-\lambda(x_\sigma)$. Thus, we only need to compute $\max_{a\in \Int(\sigma)\cap N} \dim(T_{m,a})$ when $m$ is large enough. Note that even though $\Int(\sigma)\cap N$ is an infinite set, we are actually taking maximum over the finite set of $T_m$-orbits.

By Lemma \ref{lem_lambda} we thus see if $m$ is large enough, then
\begin{equation*}
mn-\lambda(x_\sigma)=\dim(C^m)=\max_{a\in \Int(\sigma)\cap N} \dim(T_{m,a}).
\end{equation*}
Let us fix such $m$ such that, in addition, $m>n+\lambda(x_\sigma)$. From Corollary \ref{cor_orbits}, we see two cases (\ref{eqn_8.1}) and (\ref{eqn_8.2}). If $\dim(T_{m,a})\leq (m+1)n-m$, then we have
\begin{equation*}
n(m+1)-\Phi(a)\leq \dim(T_{m,a})\leq(m+1)n-c\cdot m<mn-\lambda(x_\sigma).
\end{equation*}
Therefore, replacing these $\dim(T_{m,a})$ by $n(m+1)-\Phi(a)$ does not change the maximum of $\dim(T_{m,a})$. So we get
\begin{eqnarray*}
&&\max_{a\in \Int(\sigma)\cap N} \dim(T_{m,a})\\
&=&\max_{a\in \Int(\sigma)\cap N}\Big\{(m+1)n- \Phi(a)\Big\}\\
&=&n(m+1)-\min_{a\in \Int(\sigma)\cap N}\Phi(a).
\end{eqnarray*}
The last formula gives the assertion in the theorem.
\end{proof}

\begin{cor}
\label{cor_toric}
Let $X$ be an affine toric variety over $k$ of dimension $n$ associated to a cone $\sigma$. Let $N$ be the lattice and $M$ be the dual lattice. If $\sigma$ spans $N_\mathbb{R}$ and $x_\sigma \in X$ is the torus-invariant point, the invariant $\lambda(x_\sigma)$ defined in Definition \ref{defn_lambda} is computed by the following formula
\begin{equation*}
\lambda(x_\sigma)=\min_{a\in \Int(\sigma)\cap N}\Phi(a)-n,
\end{equation*}
where the function $\Phi$ is defined in Definition \ref{defn_ph}.
\end{cor}

The following is a direct corollary of Corollary \ref{cor_toric} and Proposition \ref{fundamental_prop}.
\begin{cor}
With the same assumptions as in Corollary \ref{cor_toric}, we have
\begin{equation*}
\mld (x_\sigma;X)=\min_{a\in \Int(\sigma)\cap N}\Big\{ \min \big\{\sum_{i=1}^n \langle a,u_{i} \rangle | u_{1},\ldots,u_{n}\ \textnormal{span}\  M_\mathbb{R}\textnormal{, }u_i\in M\cap \sigma^\vee\textnormal{ for each } i\big\}\Big\},
\end{equation*}
where the second minimum is run over all linearly independent sets of vectors $\{u_1,\ldots, u_n\}$ in $M\cap \sigma^\vee$.
\end{cor}

\subsection{Examples}
$\\$
The conclusions of Theorem \ref{thm_toric} and Corollary \ref{cor_toric} involve two minima. It is not clear whether the formula can be simplified in the case of an arbitrary toric variety. However, we can simplify this formula in some special cases. Here we provide some examples of computations of the invariant $\lambda$.

\begin{exmp}
Suppose $\sigma\subset \mathbb{R}^2$ is the two dimensional cone generated by $2e_1-e_2$ and $e_2$, where $e_1$ and $e_2$ form the standard basis of $N$. Then $\sigma^\vee$ is a cone in $M_\mathbb{R}$ generated by $e_1^\ast$ and $e_1^\ast+2e_2^\ast$, where $e_1^\ast$ and $e_2^\ast$ form the dual basis. It's easy to see that $u_1=e_1^\ast$, $u_2=e_1^\ast+e_2^\ast$ and $u_3=e_1^\ast+2e_2^\ast$ form the minimal set of generators of $M\cap \sigma^\vee$.

For each $a\in N$ we write $a=(x,y)$, where $x$, $y$ are coordinates with respect to the standard basis. In order that $a\in \Int(\sigma)\cap N$, we need to have $x>0$ and $x+2y>0$. Therefore, according to Corollary \ref{cor_toric} we have
\begin{eqnarray*}
\lambda(x_\sigma)&=&\underset{x>0,x+2y>0}{\min} \min \{x+(x+y),x+(x+2y),(x+y)+(x+2y)\}-2\\
&=& \underset{x>0,x+2y>0}{\min} \min \{ 2x+y,2x+3y\}-2.
\end{eqnarray*}
It's easy to see that the minimum is equal to $0$, which is attained when $x=1$ and $y=0$, and hence $\mld(x_\sigma;X)=\dim(X)=2$.
\end{exmp}

In fact, we have the following general result:

\begin{prop}
\label{prop_simplicial_isolated}
If the torus-invariant point $x_\sigma$ is an isolated singularity of a simplicial toric variety $X$, then $\lambda(x_\sigma)=0$, and hence $\mld(x_\sigma;X)=\dim(X)$.
\end{prop}

\begin{proof}
First we claim that if $x_\sigma$ is an isolated singularity, then all facets (faces of codimension $1$) of $\sigma$ are nonsingular. Suppose that there is a proper face $\tau$ of $\sigma$ that is singular. Recall that
\begin{equation*}
O(\tau) = \Spec\ k[M\cap \tau^\perp]\cong (k^\ast)^{n-\dim(\tau)}
\end{equation*}
is the $T$-orbit that contains the distinguished point $x_\tau$. Denote by $N_\tau$ the subgroup of $N$ generated by $N\cap \tau$. Then we may choose a splitting of $N$ and write
\begin{equation*}
N=N_\tau\oplus N',\ \tau=\tau'\oplus\{0\},
\end{equation*}
where $\tau'$ is a cone in $(N_\tau)_\mathbb{R}$. Dually, we can decompose $M=M_\tau\oplus M'$. Let
\begin{equation*}
U_\tau=\Spec\ k[M\cap \tau^\vee],
\end{equation*}
and let $U_{\tau'}$ be the affine toric variety corresponding to the cone $\tau'$ and lattice $N_\tau$. With this notation, we have
\begin{equation}
\label{eqn_decomp}
U_\tau\cong \Spec\ k[M_\tau\cap \tau'^\vee]\times \Spec\ k[M']\cong U_{\tau'}\times (k^\ast)^{n-\dim(\tau)}.
\end{equation}

Note that $U_\tau$ is an open subset of $X$ that contains $O(\tau)$.  Since $\tau'$ is a singular cone, the torus-fixed point $x_{\tau'}\in U_{\tau'}$ is a singular point. In this case, the orbit $O(\tau)$, which corresponds via the above isomorphism to $\{x_{\tau'}\}\times {\Spec}\ k[M']$ is a subset of dimension
$n-\dim(\tau)$ contained in the singular locus of $X$ and that contains $0$ in its closure. This contradicts the fact that $0$ is an isolated singular point of X. So we conclude that all facets are nonsingular.

Since $X$ is simplicial, the cone $\sigma$ has only $n$ one-dimensional faces. Assume that $v_1,\ldots,v_n$ are the primitive lattice vectors on these one-dimensional faces. Then $v_1,\ldots,v_{n-1}$ span a facet of $\sigma$, and is therefore nonsingular. By applying an automorphism on $N$ one may assume that $v_1=e_1,\ldots,\ v_{n-1}=e_{n-1}$ and $v_n=a_1e_1+\cdots+a_{n-1}e_{n-1}+te_n$ with $0\leq a_i <t$. Define $a=e_1+\cdots+e_n$.

Note that $o_i:=te_i^\ast-a_ie_n^\ast$ is orthogonal to the facet spanned by $v_1,\ldots,\hat{v}_i,\ldots,v_n$ for every $i$, with $1\leq i\leq n-1$. In fact, the dual cone $\sigma^\vee$ is spanned by $o_1,\ldots,o_{n-1},e_n^\ast$. Since $\langle o_i, a\rangle=t-a_i>0$ and $\langle e_n^\ast, a\rangle=1$, $a$ is in the interior of $\sigma$.

Clearly $e_1^\ast,\ldots,e_n^\ast$ are all in the dual cone $\sigma^\vee$. In fact, each $e_i^\ast$ is on the face spanned by $o_i$ and $e_n^\ast$. Since $\varphi_a(e_i^\ast)=1$, we have $\Phi(a)\leq n$ and hence $\lambda(x_\sigma)=0$. By Proposition \ref{fundamental_prop} we get $\mld(x_\sigma;X)=\dim(X)$.
\end{proof}

\begin{cor}
If $X$ is a two-dimensional affine toric variety, then $\lambda(x_\sigma)=0$.
\end{cor}

\begin{proof}
Observe that every two-dimensional affine toric variety is simplicial, and that every facet is one-dimensional, hence nonsingular. Thus $x_\sigma$ is an isolated singularity of a simplicial toric variety. The conclusion follows immediately from Proposition \ref{prop_simplicial_isolated}.
\end{proof}

The above examples might suggest that $\lambda$ is always $0$, or $\mld$ is always equal to $\dim(X)$, for any toric variety. But this is not true in general, as we will see shortly. Now let us look at an example of a different type. We discuss this class of examples in detail in the next section (see Example \ref{exmp_binomial} for details); we refer to this section for the proof of the formula that we use.

\begin{exmp}
Let $X\subset \mathbb{A}^{n+1}$ be the hypersurface defined by the binomial function
\begin{equation*}
f=x_1x_2\cdots x_n-y^{n-1}
\end{equation*}
for some $n\geq 3$. The dimension of $X$ is $n$ while the dimension of $X_{\textnormal{sing}}$ is $n-2$. Since $X$ is Cohen-Macaulay, being a hypersurface, it follows from Serre's criterion that $X$ is normal. It is a general fact that $X$ is a toric variety if it is normal and defined by binomials (\cite[Lemma 1.1]{St95}). By applying formula (\ref{eqn_binomial}) below from Section \ref{sec5}, we immediately get $\lambda=1$.
\end{exmp}

\subsection{Extension to arbitrary closed points}

$\\$
Corollary \ref{cor_toric} gives the formula that computes the invariant $\lambda$ associated to the torus-invariant point for an affine toric variety $X$ over $k$ that corresponds to a cone that spans $N_\mathbb{R}$. Now we show how this computation is generalized to an arbitrary closed point of a toric variety $X$. We start by proving the following proposition:

\begin{prop}
\label{prop_product}
Let $X$ and $Y$ be two varieties over $k$ such that $Y$ is smooth. For any closed points $x\in X$ and $y\in Y$, the invariant $\lambda$ for the closed point $(x,y)\in X\times Y$ is equal to $\lambda(x)$.
\end{prop}

\begin{proof}
By definition, we have
\begin{equation*}
\lambda((x,y))=(\dim(X)+\dim(Y))m-\dim(\psi_m^{X\times Y} ((X\times Y)_\infty),
\end{equation*}
for $m$ large enough. According to Remark \ref{arc_product}, this is equal to
\begin{equation*}
(\dim(X)+\dim(Y))m-\dim(\psi_m^X(X_\infty))- \dim(\psi_m^Y(Y_\infty))=\lambda(x)+\lambda(y).
\end{equation*}

According to Remark \ref{relation_smooth}, $\lambda(y)=0$. Therefore, $\lambda((x,y))=\lambda(x)$.
\end{proof}

The key fact is that any closed point is in the orbit of the distinguished point of a face of $\sigma$. More precisely, let $X$ be the affine toric variety associated to a cone $\sigma$ and $O(\tau)$ be the $T$-orbit that contains the distinguished point $x_\tau$ for some face $\tau$ of $\sigma$. Then $X=\cup_\tau O(\tau)$ with $\tau$ varying over all faces of $\sigma$. We refer the reader to \cite[Chapter 3]{Ful93} for details. Then $\lambda(p)=\lambda(x_\tau)$ for every $p$ in $O(\tau)$, because there is an element $t\in T$ that maps $p$ to $x_\tau$, and such that multiplication by $t$ gives an automorphism of $X$. Therefore, it is enough to compute $\lambda(x_\tau)$, where $\tau$ is a face of $\sigma$.

Following the notation in the proof of Proposition \ref{prop_simplicial_isolated}, we have an open subset $U_\tau\cong U_{\tau'}\times (k^\ast)^{n-\dim(\tau)}$ of $X$ that contains $O(\tau)$. The point $x_\tau\in O(\tau)$ is mapped to $(x_{\tau'},\underline{1})$ by the isomorphism, where $x_{\tau'}$ is the torus-invariant point in $U_{\tau'}$. Therefore, we are reduced to computing $\lambda((x_{\tau'},\underline{1}))$ and we obtain the following corollary by using Proposition \ref{prop_product}:

\begin{cor}
With the above notation, if $X=X(\sigma)$ is an affine toric variety over $k$ of dimension $n$ and $\tau$ is a face of $\sigma$, then we have $\lambda(x_\tau)=\lambda(x_{\tau'})$. Equivalently, we have
\begin{equation*}
\mld (x_\tau;X)-n=\mld (x_{\tau'};U_{\tau'})-\dim(\tau).
\end{equation*}
\end{cor}

\section{Mather minimal log discrepancy of hypersurfaces}
\label{sec5}
This section is devoted to the computation of the Mather minimal log discrepancy of the origin on a hypersurface whose defining equation has fixed monomials and very general coefficients.

\subsection{Basic setup}
\label{sec_setup}
$\\$
In this subsection, we fix our notation and give the criterion on the defining equation such that the hypersurface is integral. Throughout the section, we assume that $X$ is a hypersurface in
\begin{equation*}
\mathbb{A}^{n+1}=\Spec\ \mathbb{C}[x_1,\ldots,x_{n+1}],
\end{equation*}
for some positive integer $n$, over the field of complex numbers. We have $\dim(X)=n$. After a change of coordinates, we may and will assume that the origin of $\mathbb{A}^{n+1}$ is contained in $X$.

Let $f$ be the defining equation of $X$ in $\mathbb{A}^{n+1}$. We can write
$f=\sum _{i=1} ^N a_{I^i} x^{I^i}$, where $I^i=(I^i_1,I^i_2,\ldots,I^i_{n+1})$ are multi-indices and $x^{I^i}$ stands for $\prod _{j=1}^{n+1} x_j ^{I_j^i}$. Suppose that all the coefficients $a_{I^i}$ are nonzero. Then $N$, the number of monomials in the polynomial $f$, is at least $2$ when $X$ is integral, unless $X$ is a coordinate plane. We denote by $\mathbb{Z}_+$ the set of positive integers and by $\mathbb{Z}_{\geq 0}$ the set of nonnegative integers.

\begin{defn}
The \emph{support} of a multi-index $I^i$ is $| I^i| := \{j | I_j^i> 0\}$. Given an $(n+1)$-tuple $\alpha=(\alpha_1,\ldots,\alpha_{n+1})\in \mathbb{Z}^{n+1}$ and a multi-index $I$, we define the \emph{product} as $\alpha \cdot I := \sum_{j=1} ^{n+1} \alpha_j I_j$. The \emph{support} of a polynomial $f=\sum _{i=1} ^N a_{I^i} x^{I^i}$, with all $a_{I^i}\neq 0$, is the set
\begin{equation*}
A=\{I^1,\ldots,I^N\}\subset (\mathbb{Z}_{\geq 0})^{n+1}.
\end{equation*}
The \emph{dimension} of $A$ is defined by $\dim(A)=\dim_\mathbb{Q} (\textnormal{Span}_\mathbb{Q}\{A-a\})$, for any $a\in A$.
\end{defn}

\begin{rem}
Clearly the dimension of a support $A$ is independent of the choice of $a$. When $X$ is integral and is not a hyperplane, the support $A$ of $f$ has at least two points, hence $\dim(A)\geq 1$.
\end{rem}

\begin{rem}
We may assume without loss of generality that
\begin{equation*}
\cup_{1\leq i\leq N} |I^i |=\{1,2,\ldots,n+1\}.
\end{equation*}
Otherwise, $X$ is the product of an affine space with a hypersurface of lower dimension. Then by Proposition \ref{prop_product}, computing $\lambda$ for the origin in $X$ is reduced to computing the corresponding $\lambda(0)$ on the lower dimensional hypersurface.
\end{rem}

\begin{rem}
If $0$ is a smooth point, then the invariant $\lambda(0)$ is trivially zero by Remark \ref{relation_smooth}. So we focus on the case where $0$ is a singular point of $X$. In particular, we assume that $X$ is not a hyperplane.
\end{rem}

In order that the hypersurface $X$ contains the origin, we require that $f$ is a polynomial in
\begin{equation*}
(x_1,x_2,\ldots,x_n)\cdot \mathbb{C}[x_1,x_2,\ldots,x_{n+1}],
\end{equation*}
or equivalently, the point $(0,0,\ldots,0)$ is not in the support of $f$. By requiring that $X$ is irreducible and is not a hyperplane, we see that $f$ is not divisible by $x_i$ for each $i$. This means that the support of $f$ contains at least one point in each coordinate plane $x_i=0$. We first characterize those $A$, such that a general polynomial with support $A$ defines an integral hypersurface. We denote by $\textnormal{conv}(A)$ the convex hull of $A$. The following result is a simplified version of \cite[Theorem 3]{Yu16}:

\begin{thm}
\label{thm_integral}
Let $R=\mathbb{C}[x_1^{\pm 1},\ldots,x_{n+1}^{\pm 1}]$ be the Laurent polynomial ring in $n+1$ variables. Then a general polynomial $f$ with support $A$ generates a proper prime ideal in $R$ if and only if one of the following holds:

\item (1) $\dim(A)\geq 2$, or

\item (2) $\dim(A)=1$ and $\textnormal{conv}(A)$ contains only two integral points.
\end{thm}

\begin{lem}
\label{lem_prime}
Let $R=\mathbb{C}[x_1^{\pm 1},\ldots,x_{n+1}^{\pm 1}]$ be the Laurent polynomial ring in $n+1$ variables and $f$ be a polynomial in $\mathbb{C}[x_1,\ldots,x_{n+1}]$ that is not divisible by any $x_i$. If $f$ generates a prime ideal in $R$, then $f$ also generates a prime ideal in $\mathbb{C}[x_1,\ldots,x_{n+1}]$.
\end{lem}

\begin{proof}
The assertion follows from the fact that
\begin{equation*}
f\cdot \mathbb{C}[x_1,\ldots,x_{n+1}]=f\cdot \mathbb{C}[x_1^{\pm 1},\ldots,x_{n+1}^{\pm 1}]\cap \mathbb{C}[x_1,\ldots,x_{n+1}].
\end{equation*}
This follows easily, using the fact that $\mathbb{C}[x_1,\ldots,x_{n+1}]$ is a UFD, from the fact that $f$ is not divisible by any $x_i$.
\end{proof}

\begin{defn}
\label{defn_integral}
A finite subset $A$ of $(\mathbb{Z}_{\geq 0})^{n+1}$ is called \emph{integral} if the following conditions hold:

\item (1) $A$ contains at least one point in each coordinate plane $x_i=0$,

\item (2) $A$ does not contain the origin $(0,\ldots,0)$, and

\item (3) $\dim(A)\geq 2$, or $\dim(A)=1$ and $\textnormal{conv}(A)$ contains only two integral points.

Let $|A|$ be the cardinality of $A$. We denote by $F(A)\subset (\mathbb{C}^\ast)^{|A|}$ the set of coefficients, such that a polynomial $f$ with support $A$ and these coefficients generates a prime ideal in the Laurent polynomial ring $R$.
\end{defn}

By Theorem \ref{thm_integral}, the set $F(A)$, for each integral subset $A\subset (\mathbb{Z}_{\geq 0})^{n+1}$, contains a nonempty open subset of $(\mathbb{C}^\ast)^{|A|}$. The following is a direct corollary of Lemma \ref{lem_prime}:

\begin{cor}
\label{cor_integral}
If $f$ is a polynomial with an integral support $A$ and coefficients in $F(A)$, then $f$ defines an integral hypersurface in $\mathbb{A}^{n+1}$ containing the origin.
\end{cor}

In what follows, we fix an integral subset $A\subset (\mathbb{Z}_{\geq 0})^{n+1}$ with cardinality $N\geq 2$, and assume that the defining equation $f$ has support $A$ and coefficients in $F(A)$. Recall that we write $f=\sum _{i=1} ^N a_{I^i} x^{I^i}$ with all $a_{I^i}\neq 0$. In this case we have $A=\{I^1,\ldots, I^N\}$.

\begin{defn}
For each positive integer $m$, we define as in the previous section the sets $C^m:=\psi_m (\pi^{-1} (0))$ in the $m^{\textnormal{th}}$ jet scheme of $X$, where $0$ is the origin of $\mathbb{A}^{n+1}$. For each $(n+1)$-tuple $\alpha\in \mathbb{Z}^{n+1}$ such that $1\leq \alpha_j\leq m$ for each $j$, we define
\begin{equation*}
C_\alpha ^m:=C^m\cap (\cap_{1\leq j\leq n+1} \cont^{\alpha_j} (x_j)_m).
\end{equation*}
In other words, $C_\alpha ^m$ is a subset of $C^m$ with prescribed order along each $x_j$.
\end{defn}

Now we fix $\alpha$ with $\alpha_j\geq 1$ for each $j$. Let $n_0(\alpha)=\min_{1\leq i \leq N} \{\alpha \cdot I^i\}$. After relabeling we may assume that the minimum is attained by precisely those $i$ with $1\leq i\leq k$ for some $k\geq 1$. Then the image of $f$ under the map
\begin{equation}
\label{eqn_base}
\mathbb{C}[x_1,x_2,\ldots,x_{n+1}]\longrightarrow (\mathbb{C}[x_j^{(s)}|1\leq j\leq n+1,s\geq \alpha_j])[\![t]\!],\ x_j\longmapsto \sum_{s=\alpha_j}^\infty x_j^{(s)}t^s,
\end{equation}
has $t$-order $\geq n_0(\alpha)$ and the coefficient of $t^{n_0(\alpha)}$ is
\begin{equation}
\label{eqn_P0}
P_0(x_1^{({\alpha_1})},x_2^{({\alpha_2})},\ldots,x_{n+1}^{({\alpha_{n+1}})}):=\sum_{i=1}^k a_{I^i}\Pi_{j=1}^{n+1} (x_j^{(\alpha_j)})^{I_j^i}.
\end{equation}
$P_0$ is an element in $\mathbb{C}[x_1^{({\alpha_1})},x_2^{({\alpha_2})},\ldots,x_{n+1}^{({\alpha_{n+1}})}]$. We define the condition $\Delta ^\alpha$, which will be used in the statements of the main result in this section, as follows:

\begin{cond}
\label{open_condition}
We say that the condition $\Delta^\alpha$ holds for $f$ if
\begin{equation*}
\bigcap_{j=1}^{n+1} \mathbb{V}\Big(\frac {\partial P_0}{\partial x_j^{(\alpha_j)}}\Big) \cap \mathbb{V}(P_0)= \emptyset\
\end{equation*}
in the torus
\begin{equation*}
(\mathbb{C}^\ast)^{n+1}= \Spec\ \mathbb{C}[x_1^{({\alpha_1})},x_2^{({\alpha_2})},\ldots, x_{n+1}^{({\alpha_{n+1}})}]_{(x_1^{({\alpha_1})},x_2^{({\alpha_2})},\ldots, x_{n+1}^{({\alpha_{n+1}})})}.
\end{equation*}
\end{cond}

\begin{defn}
The \emph{weight} of a monomial $\prod_{j=1}^{n+1}\prod_{i=1}^{b_j} x_j^{(\beta^j_i)}$ is the sum of the superscripts $\sum_{i,j}\beta_i^j$; the \emph{weight} of a polynomial in $(x_j^{(u)})_{1\leq j\leq {n+1};\ u>0}$ is the smallest weight among its monomials.
\end{defn}

\begin{rem}
\label{rem_order}
It is easy to see that for every $s$, each monomial in the coefficient of $t^s$ in the image of a polynomial under the map (\ref{eqn_base}) has weight $s$.
\end{rem}

\subsection{Main results}

\begin{lem}
\label{lem_first}
For a fixed $\alpha=(\alpha_1,\ldots,\alpha_{n+1})\in (\mathbb{Z}_+)^{n+1}$, the set
\begin{equation*}
F_\alpha :=\{ (a_{I^i})_{1\leq i\leq N} \in (\mathbb{C}^\ast)^N |\textnormal{condition}\ \Delta ^\alpha\textnormal{ is satisfied}\}
\end{equation*}
contains a nonempty open subset of $(\mathbb{C}^\ast)^N$.
\end{lem}

\begin{proof}
We use the Kleiman-Bertini Theorem in characteristic zero, which states that the general element of a linear system of divisors on a variety $Y$ is nonsingular away from the base locus of the linear system and the singular locus of $Y$.

Let $Y$ be the affine space $\mathbb{C}^{n+1}=\Spec\ \mathbb{C}[x_1^{({\alpha_1})},x_2^{({\alpha_2})},\ldots,x_{n+1}^{({\alpha_{n+1}})}]$ and $Z$ be the hypersurface in $Y$ defined by the polynomial $P_0=\sum_{i=1}^k a_{I^i}\Pi_{j=1}^{n+1} (x_j^{(\alpha_j)})^{I_j^i}$. Note that the left-hand side of the equation in condition $\Delta^\alpha$ is the singular locus of $Z$. Therefore, it suffices to show that for a general choice of coefficients, $Z$ will be nonsingular away from the coordinate planes.

The linear system of divisors $\mathbb{H}$ on Y consisting of hypersurfaces defined by polynomials of the form 
\begin{equation*}
p=\sum_{i=1}^k a_{I^i}\Pi_{j=1}^{n+1} (x_j^{(\alpha_j)})^{I_j^i}, 
\end{equation*}
for $a_{I^i}\in\mathbb{C}$, is clearly base point free away from the coordinate planes because each monomial $x^{I^i}$ is already so. By the Kleiman-Bertini Theorem, the hypersurface $Z$ is nonsingular away from the coordinate planes for a general choice of $(a_{I^i})_{1\leq i\leq N}\in(\mathbb{C}^\ast)^N$.
\end{proof}

We now study the image of $f$ under the map (\ref{eqn_base}). It suffices to study the image of each monomial of $f$. The image of $f$ is just the sum of the images of all the monomials of $f$.

\begin{lem}
\label{lem_third}
Fix $\alpha=(\alpha_1,\ldots,\alpha_{n+1})\in (\mathbb{Z}_+)^{n+1}$. For each $s\geq 1$, the coefficient of $t^s$ in the image of $x_1^{b_1}x_2^{b_2}\ldots x_{n+1}^{b_{n+1}}$ under the map (\ref{eqn_base}) is equal to
\begin{equation}
\label{eqn_monomial_image}
\sum_{c_{i,j}} \Big(\prod_{i=1}^{n+1} b_i! \cdot \prod_{j\geq \alpha_i} \frac{(x_i^{(j)})^{c_{i,j}}}{c_{i,j}!}\Big),
\end{equation}
where the sum is over all $c_{i,j}$ with $1\leq i\leq n+1$ and $j\geq \alpha_i$ such that
\begin{equation*}
\sum_{i,j}j\cdot c_{i,j}=s\textnormal{ and } \sum_{j\geq \alpha_i} c_{i,j}=b_i\textnormal{ for all }i.
\end{equation*}
\end{lem}

\begin{proof}
By considering the weight of a monomial $\prod_{i=1}^{n+1}\prod_{j\geq \alpha_i} (x_i^{(j)})^{c_{i,j}}$, it is clear that if this monomial appears with nonzero coefficient in the coefficient of $t^s$, then we have $\sum_{i,j}j\cdot c_{i,j}=s$. If such a monomial shows up in the image of $x_1^{b_1}x_2^{b_2}\ldots x_{n+1}^{b_{n+1}}$, then it clearly satisfies $\sum_{j\geq \alpha_i} c_{i,j}=b_i$ for all $i$. Moreover, it follows from the multinomial formula that if these conditions are satisfied, then the coefficient of the above monomial is
\begin{equation*}
\prod_{i=1}^{n+1} \frac{b_i!}{\prod_{j\geq \alpha_i} c_{i,j}!}
\end{equation*}
\end{proof}

This lemma shows that the images of two different monomials under the map (\ref{eqn_base}) do not mix together. Thus, the number of monomials in the coefficient of $t^s$ of the image of $f$ is the sum of the numbers of monomials in the coefficient of $t^s$ for the image of each of the monomials of $f$. Similarly, the highest superscript in the coefficient of $t^s$ of the image of $f$ is equal to the maximum of the highest superscript in the coefficient of $t^s$ that appears in the images of all the monomials of $f$.

\begin{lem}
\label{lem_forth}
Fix $\alpha=(\alpha_1,\ldots,\alpha_{n+1})\in (\mathbb{Z}_+)^{n+1}$ and a monomial $x_1^{b_1}x_2^{b_2}\ldots x_{n+1}^{b_{n+1}}$. Then for each $s\geq \sum_{i=1}^{n+1} b_i \alpha_i$, the largest superscript appearing in the coefficient of $t^s$ in the image of $x_1^{b_1}x_2^{b_2}\ldots x_{n+1}^{b_{n+1}}$ under the map (\ref{eqn_base}) is equal to $s-\mu$ for some fixed number $\mu$. Moreover, this largest superscript only appears in the monomials of the form
\begin{equation*}
(x_1^{(\alpha_1)})^{b_1}(x_2^{(\alpha_2)})^{b_2}\ldots(x_j^{(\alpha_j)})^{b_j-1}\ldots(x_{n+1}^{(\alpha_{n+1})})^{b_{n+1}} \cdot x_j^{(s-\mu)},
\end{equation*}
for some $j$ such that $\alpha_j=\max_{1\leq i\leq n+1} \alpha_i$, and we have $\mu=\sum_{i=1}^{n+1} b_i \alpha_i-\alpha_j$.
\end{lem}

\begin{proof}
According to Lemma \ref{lem_third}, the coefficient of $t^s$ in the image of the monomial $x_1^{b_1}x_2^{b_2}\ldots x_{n+1}^{b_{n+1}}$ consists of monomials of the form $\prod_{i=1}^{n+1}\prod_{j=1}^{b_i} x_i^{(\beta^i_j)}$, with $\beta^i_j\geq \alpha_i$ for every $i$ and $j$ and such that $\sum_{i,j}\beta^i_j =s$. When $s<\sum_{i=1}^{n+1}b_i\alpha_i$, there is no such monomial. Hence we require that $s\geq \sum_{i=1}^{n+1}b_i\alpha_i$. When $s=\sum_{i=1}^{n+1}b_i\alpha_i$, there is only one monomial
\begin{equation*}
(x_1^{(\alpha_1)})^{b_1}(x_2^{(\alpha_2)})^{b_2}\ldots(x_j^{(\alpha_j)})^{b_j}\ldots(x_{n+1}^{(\alpha_{n+1})})^{b_{n+1}}
\end{equation*}
in the coefficient of $t^s$. Hence the largest superscript that shows up in this case is equal to $\max_i \alpha_i$. In what follows, we assume that $s>\sum_{i=1}^{n+1}b_i\alpha_i$. Then the largest superscript that appears in the coefficient of $t^s$ is given by the following optimization problem:
\begin{eqnarray*}
 \max &&\max_{i,j}\{\beta^i_j\}\\
\textnormal{s.t.} && \beta^i_j\geq \alpha_i \textnormal{ for each } i\\
\textnormal{and}&& \sum_{i,j}\beta^i_j =s.
\end{eqnarray*}

Let $(\bar\beta^i_j)_{i,j}$ give the optimal solution to this optimization problem. We claim that there exist $i_0$ and $j_0$, with $1\leq i_0\leq n+1$ and $1\leq j_0\leq b_{i_0}$, such that $\bar\beta^i_j=\alpha_i$ if and only if $(i,j)\neq (i_0,j_0)$. First we show that if the maximum is attained by $\bar\beta^{i_0}_{j_0}$, then we have $\bar\beta^{i_0}_{j_0}>\alpha_{i_0}$. By relabeling, we assume that $\alpha_1=\max_i \alpha_i$. Consider another feasible solution $(\tilde\beta^i_j)$ to the optimization problem, with $\tilde\beta^1_1=\alpha_1+s-\sum_{i=1}^{n+1}b_i\alpha_i$ and $\tilde\beta^i_j=\alpha_i$ for $(i,j)\neq (1,1)$. Then clearly $\max_{i,j}\{\tilde\beta^i_j\}=\tilde\beta^1_1>\max_i \alpha_i$. Hence
\begin{equation*}
\bar\beta^{i_0}_{j_0}= \max_{i,j}\{\bar\beta^i_j\}\geq \max_{i,j}\{\tilde\beta^i_j\}>\alpha_{i_0}.
\end{equation*}
Now suppose contrary to our claim, that we have another pair $(i_1,j_1)\neq (i_0,j_0)$, such that $\bar\beta^{i_1}_{j_1}> \alpha_{i_1}$. We define $\beta^i_j=\bar\beta^i_j$ if $(i,j)\neq (i_0,j_0), (i_1,j_1)$, $\beta^{i_0}_{j_0}=\bar\beta^{i_0}_{j_0} +1$ and $\beta^{i_1}_{j_1}=\bar\beta^{i_1}_{j_1}-1$. Then clearly $(\beta^i_j)$ is also feasible, while $\max_{i,j}\{\beta^i_j\}>\max_{i,j}\{\bar\beta^i_j\}$. This contradicts our choice of $(\bar\beta^i_j)$.

The above discussion shows that the largest superscript that appears in a monomial in the coefficient of $t^s$ shows up only in monomials of the form
\begin{equation*}
(x_1^{(\alpha_1)})^{b_1}(x_2^{(\alpha_2)})^{b_2}\ldots(x_j^{(\alpha_j)})^{b_j-1}\ldots(x_{n+1}^{(\alpha_{n+1})})^{b_{n+1}} \cdot x_j^{(\beta_j(s))}.
\end{equation*}
Moreover, such monomials appear in the coefficient of $t^s$ for all $j$.

By considering the weight of such a monomial, we get $s=\beta_j(s)-\alpha_j+\sum_{i=1}^{n+1} b_i\alpha_i$. This implies that when $\beta_j(s)$ is the largest superscript, $\alpha_j=\max_i\alpha_i$, and
\begin{equation*}
\beta_j(s)=s-\sum_{i=1}^{n+1} b_i\alpha_i+\alpha_j.
\end{equation*}
This proves the lemma with $\mu=\sum_{i=1}^{n+1} b_i\alpha_i-\max_i\alpha_i$.
\end{proof}

\begin{rem}
With the same proof as above, one can show that for each fixed index $j$, with $1 \leq j\leq n+1$, the largest superscript for $x_j$ appearing in the coefficient of $t^s$  in the image of $x_1^{b_1}x_2^{b_2}\ldots x_{n+1}^{b_{n+1}}$ under the map (\ref{eqn_base}) appears in the monomials of the form
\begin{equation*}
(x_1^{(\alpha_1)})^{b_1}(x_2^{(\alpha_2)})^{b_2}\ldots(x_j^{(\alpha_j)})^{b_j-1}\ldots(x_{n+1}^{(\alpha_{n+1})})^{b_{n+1}} \cdot x_j^{(s-\mu)},
\end{equation*}
where $\mu=\sum_{i=1}^{n+1} b_i\alpha_i-\alpha_j$.
\end{rem}

Combining Lemma \ref{lem_third} and Lemma \ref{lem_forth}, we see that if $P=x_1^{b_1}\ldots x_{n+1}^{b_{n+1}}$ and if $\alpha_{j_0}=\max_i \alpha_i$, then for $s>\sum_{i=1}^{n+1}b_i\alpha_i$, the term with the highest superscript in the coefficient of $t^s$ for the image of $P$ under the map (\ref{eqn_base}) is equal to
\begin{equation*}
\frac{\partial P}{\partial x_{j_0}}(x_1^{(\alpha_1)},\ldots, x_{n+1}^{(\alpha_{n+1})})\cdot x_{j_0}^{(s-\mu)},
\end{equation*}
with $\mu=\sum_{i=1}^{n+1} b_i\alpha_i-\alpha_{j_0}$. When $s=\sum_{i=1}^{n+1}b_i\alpha_i$, the coefficient of $t^s$ is equal to $P(x_1^{(\alpha_1)},\ldots, x_{n+1}^{(\alpha_{n+1})})$. This is the smallest $s$ such that the coefficient of $t^s$ is nonzero. Similarly, for each fixed index $j$, the term with the highest superscript of $x_j$ is equal to
\begin{equation*}
\frac{\partial P}{\partial x_{j}}(x_1^{(\alpha_1)},\ldots, x_{n+1}^{(\alpha_{n+1})})\cdot x_{j}^{(s-\mu')},
\end{equation*}
with $\mu'=\sum_{i=1}^{n+1} b_i\alpha_i-\alpha_{j}$.

Since the image of different monomials of $f$ do not mix, by Lemma \ref{lem_forth} the coefficient of $t^s$ in the image of $f$ under the map (\ref{eqn_base}), for $s>\max_{1\leq i\leq N}\{\alpha\cdot I^i\}$, is of the form
\begin{equation}
\label{eqn_16}
\begin{split}
T_0(x_1^{({\alpha_1})},x_2^{({\alpha_2})},\ldots,x_{n+1}^{({\alpha_{n+1}})})x_{j_0}^{(s-\mu_{j_0})} \qquad\qquad\qquad\\
+ Q_s(x_j^{(\beta_j)} | \alpha_j\leq \beta_j\leq s-\mu_{j_0}; \alpha_{j_0}\leq \beta_{j_0} <s-\mu_{j_0}),
\end{split}
\end{equation}
for some index $j_0$, some $\mu_{j_0}>0$, and polynomials $T_0$ and $Q_s$. In other words, the highest superscript is $s-\mu_{j_0}$ and it is attained at the index $j_0$. Recall that $f=\sum_{i=1}^N a_{I^i} x^{I^i}$ and $n_0(\alpha)=\min_i \{\alpha\cdot I^i\}$ and this minimum is attained by all $1\leq i\leq k$. For each $i$ with $I^i_{j_0}>0$, we compute the product $\alpha\cdot I^i$ and define
\begin{equation*}
n_0(\alpha)':=\min_{I^i_{j_0}>0} \{\alpha\cdot I^i\} \textnormal{ and } \sigma:=\{1\leq i\leq N|I^i_{j_0}>0,\ \alpha\cdot I^i=n_0(\alpha)'\}.
\end{equation*}
Clearly we have $n_0(\alpha)'\geq n_0(\alpha)$. According to the discussion for a monomial above, $n_0(\alpha)'$ is the smallest integer such that the coefficient of $t^{n_0(\alpha)'}$ contains a monomial divisible by $x_{j_0}^{(q)}$ for some $q$ (in fact it is divisible by $x_{j_0}^{(\alpha_{j_0})}$). Moreover, the coefficient of $t^{n_0(\alpha)'}$ is equal to
\begin{equation}
\label{eqn_P1}
P_1(x_1^{({\alpha_1})},x_2^{({\alpha_2})},\ldots,x_{n+1}^{({\alpha_{n+1}})}) +\textnormal{ other terms without } x_{j_0},
\end{equation}
where $P_1(x_1,x_2,\ldots,x_{n+1})=\sum_{i\in \sigma} a_{I^i}x^{I^i}$. Clearly if $n_0(\alpha)'=n_0(\alpha)$, then $\sigma \subset \{1,2,\ldots,k\}$. Otherwise, if $n_0(\alpha)'>n_0(\alpha)$, then $\sigma \subset \{k+1,\ldots,N\}$. By the discussion for the case of a monomial, we have
\begin{equation}
\label{eqn_T0}
T_0(x_1^{({\alpha_1})},x_2^{({\alpha_2})},\ldots,x_{n+1}^{({\alpha_{n+1}})}) =\frac {\partial P_1(x_1^{({\alpha_1})},x_2^{({\alpha_2})},\ldots,x_{n+1}^{({\alpha_{n+1}})})}{\partial x_{j_0}^{(\alpha_{j_0})}}.
\end{equation}

For each fixed index $j$, similar arguments for the highest superscript of $x_j$ also hold.

\begin{rem}
Consider the weight of the first term in equation (\ref{eqn_P1}), we get $I^i\cdot \alpha =n_0(\alpha)'$ for each $i\in \sigma$. Hence each monomial in $T_0(x_1^{({\alpha_1})},x_2^{({\alpha_2})},\ldots,x_{n+1}^{({\alpha_{n+1}})})$ has weight $n_0(\alpha)'-\alpha_{j_0}=I^i\cdot \alpha -\alpha_{j_0}$ for every $i\in \sigma$. Consider the weight of the first term in equation (\ref{eqn_16}) we get $s=I^i\cdot \alpha-\alpha_{j_0}+s-\mu_{j_0}$ for each $i\in \sigma$, or equivalently, $\mu_{j_0}=I^i\cdot \alpha-\alpha_{j_0}$. But $s-\mu_{j_0}$ is the highest superscript appearing in the coefficient of $t^s$. Therefore, we have
\begin{equation}
\label{eqn_mu}
\mu_{j_0}=\underset{1\leq j\leq n+1\textnormal{ with }  I_j^i>0}{\min_{1\leq i\leq N}} \{I^i\cdot \alpha-\alpha_j \},
\end{equation}
and the minimum is attained when $i\in \sigma$ and $j=j_0$. The condition $I_j^i>0$ is equivalent to $\frac{\partial P_1}{\partial x_{j}^{(\alpha_{j})}} \neq 0$.
\end{rem}

Recall that $C^m=\psi_m(\pi^{-1}(0))$ is a contact locus in $X_m$ and
\begin{equation*}
C_\alpha ^m:=C^m\cap (\cap_{1\leq j\leq n+1} \cont^{\alpha_j} (x_j)_m).
\end{equation*}
is a contact locus in $C^m$ for each $(n+1)$-tuple $\alpha$. The following lemma gives an upper bound to the dimension of $C^m_\alpha$.

\begin{lem}
\label{lem_second}
Fix $\alpha=(\alpha_1,\ldots,\alpha_{n+1})\in (\mathbb{Z}_+)^{n+1}$. Let $m$ be an integer such that $\alpha_j\leq m$ for each $j$. If $\min_{1\leq i \leq N} \{\alpha \cdot I^i\}$ is attained by a unique $i$, then $C_\alpha^m=\emptyset$. If $\min_{1\leq i \leq N} \{\alpha \cdot I^i\}$ is attained by at least two different $i$'s and if $(a_{I^i})_{1\leq i\leq N} \in F_\alpha\cap F(A)$, where $F_\alpha$ is as defined in Lemma \ref{lem_first} and $F(A)$ is as defined in Definition \ref{defn_integral}, then we have
\begin{equation}
\label{eqn_14}
\dim C_\alpha^m \leq mn-\sum_{j=1}^{n+1} (\alpha_j-1)-1+ \min_{1\leq i\leq N} \{I^i\cdot \alpha \}-\underset{1\leq j\leq n+1\textnormal{ with }  I_j^i>0}{\min_{1\leq i\leq N}} \{I^i\cdot \alpha-\alpha_j \},
\end{equation}
for all $m$ large enough.
\end{lem}

\begin{rem}
\label{rem_feasible}
An $(n+1)$-tuple $\alpha\in (\mathbb{Z}_+ )^{n+1}$ is called  \emph{feasible} if $\min_{1\leq i \leq N} \{\alpha \cdot I^i\}$ is attained by at least two different $i$'s. Otherwise, it is called \emph{non-feasible}. According to Lemma \ref{lem_second}, $C^m_\alpha=\emptyset$ if $\alpha$ is non-feasible.
\end{rem}

\begin{proof}
Consider the affine space
\begin{equation*}
\Spec\ \mathbb{C}[x_1^{(\alpha_1)}, x_1^{(\alpha_1+1)}, \ldots, x_1^{(m)}, \ldots , x_{n+1}^{(\alpha_{n+1})}, x_{n+1}^{(\alpha_{n+1}+1)}, \ldots, x_{n+1}^{(m)}]=\mathbb{A}^{m(n+1)-\sum_{i=1}^{n+1} (\alpha_i-1)}.
\end{equation*}
 Clearly $x_1^{({\alpha_1})},x_2^{({\alpha_2})},\ldots,x_{n+1}^{({\alpha_{n+1}})}$ are regular functions on this affine space. We denote by $U_m$ the open subset of
 \begin{equation*}
 \mathbb{A}^{m(n+1)-\sum_{i=1}^{n+1} (\alpha_i-1)}
 \end{equation*}
 where $x_1^{({\alpha_1})},x_2^{({\alpha_2})},\ldots,x_{n+1}^{({\alpha_{n+1}})}$ do not vanish. Let $m_0=\max_{1\leq j\leq n+1}\{\alpha_j\}$. When $m\geq m_0$, $C^m_\alpha$ is naturally embedded in $U_m$.

Pick an arc $\gamma$ in $X_\infty\cap(\cap_{1\leq i\leq n+1}\cont ^{\alpha_i}(x_i))$. Then $\gamma$ is represented by a homomorphism of $\mathbb{C}$-algebras \begin{equation}
\label{eqn_base2}
\gamma^\ast: \mathbb{C}[x_1,\ldots,x_{n+1}]\longrightarrow \mathbb{C}[\![t]\!],\ \gamma^\ast(x_i)= \sum_{j=\alpha_i}^\infty x_i^{(j)}t^j,
\end{equation}
such that $\gamma^\ast(f)=0$ and $x_i^{(\alpha_i)}\neq 0$ for each $i$. Let us write $G_s$ for the coefficient of $t^s$ in $\gamma^\ast(f)$. Then by definition (equation (\ref{eqn_P0})), we have $G_{n_0(\alpha)}=P_0(x_1^{({\alpha_1})},x_2^{({\alpha_2})},\ldots,x_{n+1}^{({\alpha_{n+1}})})$. Hence, as long as $m>n_0(\alpha)$, we have $C^m_\alpha \subset \V(P_0(x_1^{({\alpha_1})},x_2^{({\alpha_2})},\ldots,x_{n+1}^{({\alpha_{n+1}})}))$. But when $\min_{1\leq i \leq N} \{\alpha \cdot I^i\}$ is attained by a unique $i$, $P_0$ is a monomial. Thus, we have $\V(P_0(x_1^{({\alpha_1})},x_2^{({\alpha_2})},\ldots,x_{n+1}^{({\alpha_{n+1}})})) =\emptyset$ in $U_m$, which implies that $C^m_\alpha=\emptyset$. This shows that $C^m_\alpha=\emptyset$ if $\alpha$ is non-feasible. In what follows, we assume that $\alpha$ is feasible and $(a_{I_i})_i\in F_\alpha\cap F(A)$, and show that in such a case, $C_\alpha^m$ is a finite union of locally closed subsets of $U_m$, all of them having dimension less than or equal to the right-hand side of (\ref{eqn_14}).

For each $m\geq m_0$, we define $A_0:=C^m_\alpha \backslash \V (T_0)$.  Since $(a_{I_i})_i\in F_\alpha$, we can find $j$ such that
\begin{equation*}
\frac{\partial P_0}{\partial x_j^{(\alpha_j)}}(x_1^{({\alpha_1})},x_2^{({\alpha_2})},\ldots,x_{n+1}^{({\alpha_{n+1}})})\neq 0.
\end{equation*}
Then for each $s\geq 1$, consider the highest superscript of $x_j$ in the coefficient of $t^{n_0(\alpha)+s}$ and we get
\begin{equation}
\label{eqn_11}
G_{n_0(\alpha)+s} = \frac{\partial P_0}{\partial x_j^{(\alpha_j)}}(x_1^{({\alpha_1})},x_2^{({\alpha_2})},\ldots,x_{n+1}^{({\alpha_{n+1}})})\cdot x_j^{(\alpha_j+s)} + R_s,
\end{equation}
where $R_s$ is a polynomial in $\{x_i^{(t_i)}|\alpha_i\leq t_i\leq \alpha_i+s\textnormal{ for all } i\neq j\textnormal{, }\alpha_j\leq t_j<\alpha_j+s\}$. For each $s\geq 1$, if we consider the highest superscript among all $x_i$, for $1\leq i\leq n+1$, in the coefficient of $t^{n_0(\alpha)'+s}$, we get
\begin{equation}
\label{eqn_12}
G_{n_0(\alpha) '+s} = T_0(x_1^{({\alpha_1})},x_2^{({\alpha_2})},\ldots,x_{n+1}^{({\alpha_{n+1}})})\cdot x_{j_0}^{(\alpha_{j_0}+s)} + \tilde{R}_s,
\end{equation}
where $\tilde{R}_s$ is a polynomial in $\{x_i^{(t_i)}|\alpha_i\leq t_i\leq \alpha_{j_0}+s\textnormal{ for all } i\neq j_0\textnormal{, }\alpha_{j_0}\leq t_{j_0}<\alpha_{j_0}+s\}$.

Note that
\begin{equation*}
G_{n_0(\alpha)'+m-\alpha_{j_0}}=T_0(x_1^{({\alpha_1})},x_2^{({\alpha_2})},\ldots,x_{n+1}^{(\alpha_{n+1})})\cdot x_{j_0}^{(m)}+\tilde{R}_{m-n_0(\alpha)'}.
\end{equation*}
Every variable in the above expression of $G_{n_0(\alpha)'+m-\alpha_{j_0}}$ has a superscript $\leq m$. The same holds for $G_k$, with $n_0(\alpha)\leq k< n_0(\alpha)'+m-\alpha_{j_0}$. Hence each $\V(G_k)$, with $n_0(\alpha)\leq k\leq n_0(\alpha)'+m-\alpha_{j_0}$, can be considered as a closed subset of $U_m$.

We claim that if $m>n_0(\alpha)'$, then
\begin{equation*}
A_0 = \V(G_{n_0(\alpha)}, G_{n_0(\alpha)+1},\ldots, G_{n_0(\alpha)'+m-\alpha_{j_0}}) \backslash \V(T_0)\subset U_m.
\end{equation*}
In fact, if we embed $A_0$ naturally in $A_\infty:=\Spec\ \mathbb{C} [x_i^{(s_i)}|1\leq i\leq n+1,s_i\geq \alpha_i]$, and consider each $\V(G_k)$ as a subset of $A_\infty$, then we have
\begin{equation*}
A_0= U_m\cap (\cap_{s\geq 0}\V(G_{n_0(\alpha)+s}))\backslash \mathbb{V}(T_0).
\end{equation*}
Hence, $A_0$ is contained in $\V(P_0, G_{n_0(\alpha)+1},\ldots, G_{n_0(\alpha)'+m-\alpha_{j_0}}) \backslash \V(T_0)\subset U_m$.

On the other hand, for each $s\geq 1$, if we fix
\begin{equation*}
\{x_i^{(t_i)}|\alpha_i\leq t_i\leq \alpha_i+s\textnormal{ for all } i\neq j\textnormal{, }\alpha_j\leq t_j<\alpha_j+s\}
\end{equation*}
 such that $\frac{\partial P_0}{\partial x_j^{(\alpha_j)}}(x_1^{({\alpha_1})},x_2^{({\alpha_2})},\ldots,x_{n+1}^{({\alpha_{n+1}})})\neq 0$, then the equation $G_{n_0(\alpha)+s}=0$ has a unique solution for $x_j^{(\alpha_j+s)}$. Similarly, if we fix
 \begin{equation*}
 \{x_i^{(t_i)}|\alpha_i\leq t_i\leq \alpha_{j_0}+s\textnormal{ for all } i\neq j_0\textnormal{, }\alpha_{j_0}\leq t_{j_0}<\alpha_{j_0}+s\}
 \end{equation*}
 such that $T_0(x_1^{({\alpha_1})},x_2^{({\alpha_2})},\ldots,x_{n+1}^{({\alpha_{n+1}})}) \neq 0$, then the equation $G_{n_0(\alpha)'+s}=0$ has a unique solution for $x_{j_0}^{(\alpha_{j_0}+s)}$. The existence of solutions for $G_{n_0(\alpha)+s}=0$ and $G_{n_0(\alpha)'+s}=0$, for each $s\geq 1$, shows that every element in $U_m\cap\V(P_0, G_{n_0(\alpha)+1},\ldots, G_{n_0(\alpha)'+m-\alpha_{j_0}}) \backslash \V(T_0)$ can be lifted to an element in $X_\infty$, and hence contained in $A_0$. Moreover, we see that each equation $G_{n_0(\alpha)+s}=0$ or $G_{n_0(\alpha)'+s}=0$ cuts down the dimension exactly by $1$. This  shows that the codimension of $A_0$ in $U_m$ is exactly the number of equations unless $A_0=\emptyset$. We conclude that if $A_0\neq \emptyset$, then
\begin{eqnarray*}
 \dim(A_0)&=&
 m(n+1)-\sum_{i=1}^{n+1} (\alpha_i-1)+n_0(\alpha)-(n_0(\alpha)'-\alpha_{j_0})-1-m\\
&=& mn-\sum_{i=1}^{n+1} (\alpha_i-1)-1+n_0(\alpha)-\mu_{j_0}.
\end{eqnarray*}

Now suppose that $\psi_m(\gamma)\in C^m_\alpha\cap \V(T_0)$. Then the first term on the right-hand side of equation (\ref{eqn_12}) vanishes. If we delete the first term and rearrange the equation to get a new highest-superscript term, we claim that for $s$ sufficiently large (independent of $m$), the equation (\ref{eqn_12}) becomes
\begin{equation*}
G_{n_0(\alpha)'+s}=T_1\cdot x_{j_1}^{(s-\mu_1)}+ \textnormal{Remaining Terms without }x_{j_1}^{(s-\mu_1)}
\end{equation*}
for some polynomial $T_1$ and some number $\mu_1$, with $s-\mu_1$ being the highest superscript.

To prove this, let us consider the monomials in the expression of $G_{n(\alpha)'+s}$ after deleting $T_0\cdot x_{j_0}^{(\alpha_{j_0}+s)}$. According to the proof of Lemma \ref{lem_forth}, they are of the form $\prod_{j=1}^{n+1}\prod_{k=1}^{I^i_j} x_i^{(\beta^j_k)}$ for some $i$, with $\beta^j_k\geq \alpha_j$ and $\sum_{j,k}\beta^j_k=n_0(\alpha)'+s$. Hence, the number
\begin{equation*}
\max_{s\geq 1}\{\max_{j,k}\{\beta^j_k\}-s\}
\end{equation*}
is bounded above. In fact, it is bounded above by $\alpha_{j_0}$. Suppose that the maximum is attained by some $s_0\geq 1$ and $(\bar\beta^j_k)_{j,k}$, with $\bar\beta^{j_1}_1=\max_{j,k}\{\bar\beta^j_k\}$. In other words, the highest superscript appears in the monomial
\begin{equation*}
x_1^{(\bar\beta^1_1)}\ldots x_1^{\big(\bar\beta^1_{I^i_1}\big)}\ldots \widehat{x_{j_1}^{(\bar\beta^{j_1}_1)} }\ldots x_{n+1}^{\big(\bar\beta^{n+1}_{I^i_{n+1}}\big)}\cdot x_{j_1}^{(\bar\beta^{j_1}_1)}.
\end{equation*}
On the other hand, for each $s\geq s_0$, $G_{n_0(\alpha)'+s}$ contains the monomial
\begin{equation*}
x_1^{(\bar\beta^1_1)}\ldots x_1^{\big(\bar\beta^1_{I^i_1}\big)}\ldots \widehat{x_{j_1}^{(\bar\beta^{j_1}_1)} }\ldots x_{n+1}^{\big(\bar\beta^{n+1}_{I^i_{n+1}}\big)}\cdot x_{j_1}^{(\bar\beta^{j_1}_1+s-s_0)}.
\end{equation*}
This shows that $\max_{s\geq 1}\{\max_{j,k}\{\beta^j_k\}-s\}$ is attained by all $s\geq s_0$ and the same $j_1$. Hence the highest superscript in $G_{n_0(\alpha)'+s}$ is equal to $s-\mu_1$ for some fixed $\mu_1$ and for $s$ sufficiently large. We also see that if $M\cdot x_{j_1}^{(s-\mu_1)}$ is a monomial in $G_{n_0(\alpha)'+s}$ that contains the highest superscript, then $M\cdot x_{j_1}^{(s'-\mu_1)}$ is a monomial in $G_{n_0(\alpha)'+s'}$ that contains the highest superscript for each $s'>s$. Moreover, for every such monomial, the weight of $M$ is equal to $n_0(\alpha)'+\mu_1$. There are only finitely many monomials with a fixed weight. Hence, we get
\begin{equation*}
G_{n_0(\alpha)'+s}=T_1\cdot x_{j_1}^{(s-\mu_1)}+ \textnormal{Remaining Terms without }x_{j_1}^{(s-\mu_1)}
\end{equation*}
for $s$ sufficiently large and for a fixed polynomial $T_1$. Clearly, we have $s-\mu_1\leq \alpha_{j_0}+s$, or equivalently, $\mu_1\geq -\alpha_{j_0}$.

Let $m_1$ be $\geq $ the largest superscript appearing in $T_1$ and $m_1\geq m_0$. Then for each $m\geq m_1$, we define $A_1:=C^m_\alpha\cap \V(T_0)\backslash \V(T_1)\subset U_m$. With the same analysis as above we can show that
\begin{equation*}
A_1= \V(T_0,G_{n_0(\alpha)},G_{n_0(\alpha)+1},\ldots,G_{n_0(\alpha)' +\mu_1+m}) \backslash \V(T_1)
\end{equation*}
and that each $G_i$ cuts down dimension exactly by $1$. Hence either $A_1=\emptyset$ or
\begin{eqnarray*}
\dim(A_1)&\leq&
m(n+1)-\sum_{i=1}^{n+1} (\alpha_i-1)+n_0(\alpha)-1-n_0(\alpha)'-\mu_1-m\\
&\leq& mn-\sum_{i=1}^{n+1} (\alpha_i-1)-1+n_0(\alpha)-(n_0(\alpha)'-\alpha_{j_0})\\
&=& mn-\sum_{i=1}^{n+1} (\alpha_i-1)-1+n_0(\alpha)-\mu_{j_0}.
\end{eqnarray*}

Inductively, suppose we have $A_{k}=C^m_\alpha\cap (\cap_{l\leq k-1} \V(T_l))\backslash \V(T_k)$ for each $m\geq m_k$, for some number $m_k\geq \max_{0\leq i\leq k-1}\{m_i\}$, and when $\psi_m(\gamma)\in A_k$ we have
\begin{equation}
\label{eqn_17}
G_{n_0(\alpha)'+s}=T_k\cdot x_{j_k}^{(s-\mu_k)}+ \textnormal{Remaining Terms},
\end{equation}
for $s$ sufficiently large and some number $\mu_k\geq -\alpha_{j_0}$, where $s-\mu_k$ is the highest superscript.

Now suppose that $\psi_m(\gamma)\in C^m_\alpha\cap (\cap_{l\leq k} \V(T_l))$, then the first term of the right-hand side of equation (\ref{eqn_17}) vanishes. If we delete the first term and rearrange the equation to get a new highest-superscript term, with the same proof as above we can show that
\begin{equation*}
G_{n_0(\alpha)'+s}=T_{k+1}\cdot x_{j_{k+1}}^{(s-\mu_{k+1})}+ \textnormal{Remaining Terms},
\end{equation*}
for $s$ sufficiently large (independent of $m$). Clearly, we have $\mu_{k+1}\geq \mu_k\geq -\alpha_{j_0}$.

Let $m_{k+1}$ be $\geq $ the highest superscript in $T_{k+1}$ and $m_{k+1}\geq m_k$. For each $m\geq m_{k+1}$, we define $A_{k+1}:=C^m_\alpha\cap (\cap_{l\leq k} \V(T_l))\backslash \V(T_{k+1})$. With the same analysis as in the case when $k=0$ we can show that
\begin{equation*}
A_{k+1}= \V(T_0,\ldots,T_k,G_{n_0(\alpha)},G_{n_0(\alpha)+1},\ldots,G_{n_{k+1}'+m}) \backslash \V(T_{k+1})
\end{equation*}
and that each $G_i$ cuts down dimension exactly by $1$. Hence either $A_{k+1}=\emptyset$ or
\begin{eqnarray*}
\dim(A_{k+1})&\leq&
m(n+1)-\sum_{i=1}^{n+1} (\alpha_i-1)+n_0(\alpha)-1-n_0(\alpha)'-\mu_{k+1}-m\\
&\leq& mn-\sum_{i=1}^{n+1} (\alpha_i-1)-1+n_0(\alpha)-\mu_{j_0}.
\end{eqnarray*}

We claim that there can be only finitely many such steps. First note that the highest superscript decreases, or equivalently, the number $\mu_k$ increases, by at least $1$ as $k$ increases by $n+1$ because we must have used the same subscript $j_k$ during $n+2$ steps. Second, the decrease in highest superscript must eventually stop because $G_{n_0(\alpha)+s}$ contains the term $\frac{\partial P_0}{\partial x_j^{(\alpha_j)}}(x_1^{({\alpha_1})},x_2^{({\alpha_2})},\ldots,x_{n+1}^{({\alpha_{n+1}})})\cdot x_j^{(\alpha_j+s)}$ for some $j$, with $\frac{\partial P_0}{\partial x_j^{(\alpha_j)}}(x_1^{({\alpha_1})},x_2^{({\alpha_2})},\ldots,x_{n+1}^{({\alpha_{n+1}})}) \neq 0$, and for every $s\geq 1$. Hence, for each $m$ large enough, we can decompose $C^m_\alpha$ into a finite union $\cup_{i\geq0} A_i$, where each $A_i$ either is empty or has dimension less than or equal to the number in the lemma. This completes the proof.
\end{proof}

\begin{rem}
\label{rem_equality}
From the proof of Lemma \ref{lem_second} we see that for a fixed feasible $\alpha$, coefficients $(a_{I_i})_i\in F_\alpha\cap F(A)$ and $m$ large enough, if $A_0\neq \emptyset$, then we have
\begin{equation}
\label{eqn_18}
\dim (C_\alpha^m)= mn-\sum_{j=1}^{n+1} (\alpha_j-1)-1+ \min_{1\leq i\leq N} \{I^i\cdot \alpha \}-\underset{1\leq j\leq n+1\textnormal{ with }  I_j^i>0}{\min_{1\leq i\leq N}} \{I^i\cdot \alpha-\alpha_j \}.
\end{equation}
We also see that $G_{n_0(\alpha)+s}=0$ for each $s\geq 1$ has a solution as long as $T_0\neq 0$. Therefore, $A_0\neq \emptyset$ for $m$ large enough if and only if $\V(P_0)\backslash \V(T_0)\neq \emptyset$ in the torus 
\begin{equation*}
(\mathbb{C}^\ast)^{n+1}= \Spec\ \mathbb{C}[x_1^{({\alpha_1})},x_2^{({\alpha_2})},\ldots, x_{n+1}^{({\alpha_{n+1}})}]_{(x_1^{({\alpha_1})},x_2^{({\alpha_2})},\ldots, x_{n+1}^{({\alpha_{n+1}})})}.
\end{equation*}
Note that the condition $\V(P_0)\backslash \V(T_0)\neq \emptyset$ is independent of $m$.
\end{rem}

We also need the following result:
\begin{prop}
\label{prop_finite_components}
(\cite[Proposition 3.5]{dFEI08})
If $X$ is a variety over $k$, then the number of irreducible components of a cylinder on $X_\infty$ is finite.
\end{prop}

Recall that if we fix an integral support $A=\{I^1,\ldots, I^N\}$, then $F(A)$ and $F_\alpha$ are subsets of $(\mathbb{C}^\ast)^N$ defined in Definition \ref{defn_integral} and Lemma \ref{lem_first} respectively. Each of them contains an open dense subset of $(\mathbb{C}^\ast)^N$. An $(n+1)$-tuple $\alpha=(\alpha_1,\ldots, \alpha_{n+1})$ is called feasible if $\min_{1\leq i\leq N}\{\alpha\cdot I^i\}$ is attained by at least two different $i$'s. For each feasible $\alpha$, we define polynomials $P_0$ by equation (\ref{eqn_P0}) and $T_0$ by equation (\ref{eqn_16}). Using the above lemmas we obtain the following theorem:
\begin{thm}
\label{thm_hypersurface}
Let $A=\{I^1,\ldots, I^N\}$ be an integral support.
If
\begin{equation*}
(a_{I^i})_{1\leq i\leq N}\in \bigcap_{\textnormal{feasible }\alpha} F_\alpha\cap F(A)
\end{equation*}
and $X$ is the hypersurface in $\mathbb{A}^{n+1}$ defined by $f=\sum_{i=1}^N a_{I^i} x^{I^i}$, then $X$ is an integral hypersurface containing the origin $0$ and the invariant $\lambda$ (defined in Definition \ref{defn_lambda}) for the origin satisfies
\begin{equation}
\label{eqn_hypersurface}
\lambda(0) \geq \min \{\sum_{j=1}^{n+1} (\alpha_j-1)+1-\underset{1\leq i\leq N}{\min} \{I^i\cdot \alpha \}+\underset{1\leq j\leq n+1\textnormal{ with }  I_j^i>0}{\min_{1\leq i\leq N}} \{I^i\cdot \alpha-\alpha_j \} \},
\end{equation}
where the first minimum is taken over all feasible $(n+1)$-tuples $\alpha$.

Moreover, assume the first minimum is attained at some feasible $\alpha$. If for this $\alpha$, we have $\V(P_0)\backslash \V(T_0)\neq \emptyset$ in the torus 
\begin{equation*}
(\mathbb{C}^\ast)^{n+1}= \Spec\ \mathbb{C}[x_1^{({\alpha_1})},x_2^{({\alpha_2})},\ldots, x_{n+1}^{({\alpha_{n+1}})}]_{(x_1^{({\alpha_1})},x_2^{({\alpha_2})},\ldots, x_{n+1}^{({\alpha_{n+1}})})},
\end{equation*}
then the inequality (\ref{eqn_hypersurface}) is in fact an equality.
\end{thm}

\begin{proof}
Let $f$ be fixed with $(a_{I^i})_{1\leq i\leq N}\in \cap_\alpha F_\alpha\cap F(A)$. By Corollary \ref{cor_integral}, $X$ is an integral hypersurface containing the origin $0$. According to Proposition \ref{prop_finite_components}, $\pi^{-1}(0)$ contains only finitely many irreducible components $C_1,\ldots,C_p,\ Z_1,\ldots,Z_q$, where each $C_j$ is thin and each $Z_i$ is fat. We have $\dim (C^m)=mn-\lambda(0)$ when $m$ is large enough. For each thin irreducible component $C_j$ of $\pi^{-1}(0)$, however, by Lemma \ref{lem_4.3} we see that $\dim(\psi_m(C_j))\leq (m+1)(n-1)$. Thus, for $m$ large enough, we have
\begin{equation*}
\dim (C^m)=\underset{1\leq i\leq q}{\max} \dim(\psi_m(Z_i)).
\end{equation*}

By Lemma \ref{lem_4.3}, the fibers of $\psi_{m+1}(\pi^{-1} (x))\rightarrow \psi_m(\pi^{-1} (x))$ have dimension $\leq n$. Since $\dim (C^m)=mn-\lambda(0)$ for every $m$ large enough, we have
\begin{equation*}
\dim (C^{m+1})= \dim(C^m)+n\textnormal{ for }m\gg 0.
\end{equation*}
 This also implies that there is some $i$ such that
\begin{equation*}
\dim (C^m)=\dim(\psi_m(Z_i)) \textnormal{ for all }m\gg 0.
\end{equation*}
In fact, pick a positive integer $M$ such that $\dim (C^{m+1})= \dim(C^m)+n$ for all $m\geq M$. If $\dim(\psi_m(Z_i))< \dim (C^m)$ for some $m\geq M$, then
\begin{eqnarray*}
\dim(\psi_{m+k}(Z_i))&\leq& \dim(\psi_m(Z_i))+nk\\
&<&\dim (C^m)+nk\\
&=& \dim (C^{m+k}),
\end{eqnarray*}
for every $k\geq 0$. It follows that if there exists $m_i\geq M$ for each $i$ such that $\dim(\psi_{m_i}(Z_i))< \dim (C^{m_i})$, we have $\dim(C^m)> \max_{1\leq i\leq q}\dim(\psi_m(Z_i))$ when $m> \max_{1\leq i\leq q} \{m_i\}$, a contradiction. Therefore, by relabeling we may assume that
\begin{equation*}
\dim (C^m)=\dim(\psi_m(Z_1))\textnormal{ for all }m\geq M.
\end{equation*}

Since $X=\V(f)$ is irreducible and is not a hyperplane, we have $\V(x_i,f)\subsetneq \V(f)$ for each $i$. This implies that the fat component $Z_1$ is not contained in $\cont^\infty(x_i)$ for each $i$, or equivalently, $Z_1$ does not have infinite order along any $x_i$. Choose an $(n+1)$-tuple $\alpha'=(\alpha_1',\ldots, \alpha_{n+1}')\in (\mathbb{Z}_{+})^{n+1}$ with
\begin{equation*}
\alpha_i'=\min\{\ord_\gamma(x_i)|\gamma\in Z_1\}\textnormal{, } 1\leq i\leq n+1.
\end{equation*}
Then if $m\geq \max\{M,\alpha_1',\ldots, \alpha_{n+1}'\}$, $C^m_{\alpha'}$ contains a dense open subset of $\psi_m(Z_1)$, hence $\dim(C^m)=\dim(C^m_{\alpha'})$. By applying Lemma \ref{lem_second}, we get for $m\gg 0$,
\begin{equation*}
\dim(C^m)\leq mn-\sum_{j=1}^{n+1} (\alpha_j'-1)-1+\underset{1\leq i\leq N}{\min} \{I^i\cdot \alpha' \}-\underset{1\leq j\leq n+1\textnormal{ with }  I_j^i>0}{\min_{1\leq i\leq N}}\{I^i\cdot \alpha'-\alpha_j' \}.
\end{equation*}
Therefore, for $m\gg 0$ we have
\begin{eqnarray*}
\lambda(0)&=& mn-\dim(C^m)\\
&\geq& \sum_{j=1}^{n+1} (\alpha_j'-1)+1-\underset{1\leq i\leq N}{\min} \{I^i\cdot \alpha' \}+\underset{1\leq j\leq n+1\textnormal{ with }  I_j^i>0}{\min_{1\leq i\leq N}}\{I^i\cdot \alpha'-\alpha_j' \}\\
&\geq& \underset{\textnormal{feasible }\alpha}{\min} \big\{\sum_{j=1}^{n+1} (\alpha_j-1)+1-\underset{1\leq i\leq N}{\min} \{I^i\cdot \alpha \}+\underset{1\leq j\leq n+1\textnormal{ with }  I_j^i>0}{\min_{1\leq i\leq N}}\{I^i\cdot \alpha-\alpha_j \} \big\}.
\end{eqnarray*}

Now suppose the first minimum in (\ref{eqn_hypersurface}) is attained at some feasible $\alpha$ and that for this $\alpha$ we have
\begin{equation*}
\dim (C_\alpha^m)= mn-\sum_{j=1}^{n+1} (\alpha_j-1)-1+ \min_{1\leq i\leq N} \{I^i\cdot \alpha \}-\underset{1\leq j\leq n+1\textnormal{ with }  I_j^i>0}{\min_{1\leq i\leq N}} \{I^i\cdot \alpha-\alpha_j \}.
\end{equation*}
Since $C^m_\alpha\subset C^m$, we have $\dim(C_\alpha^m)\leq \dim(C^m)$. On the other hand, we have $\dim(C^m_\alpha)\geq \dim(C^m_{\alpha'})=\dim(C^m)$ by the choice of $\alpha$. This shows that
\begin{eqnarray*}
\lambda(0)&=&mn-\dim(C^m_\alpha)\\
&=&\sum_{j=1}^{n+1} (\alpha_j-1)+1-\underset{1\leq i\leq N}{\min} \{I^i\cdot \alpha \}+\underset{1\leq j\leq n+1\textnormal{ with }  I_j^i>0}{\min_{1\leq i\leq N}}\{I^i\cdot \alpha-\alpha_j \}.
\end{eqnarray*}

In other words, we obtain an equality if $\dim(C^m_\alpha)$ attains the upper bound in the statement of Lemma \ref{lem_second} for some feasible $\alpha$ where the first minimum in (\ref{eqn_hypersurface}) is attained. According to Remark \ref{rem_equality}, this happens when $\V(P_0)\backslash \V(T_0)\neq \emptyset$ for such an $\alpha$.
\end{proof}

Combining Lemma \ref{lem_first}, Theorem \ref{thm_hypersurface} and Proposition \ref{fundamental_prop}, we get the following corollary:

\begin{cor}
\label{cor_hypersurface}
Let $A=\{I^1,\ldots,I^N\}\subset (\mathbb{Z}_{\geq 0})^{n+1}$ be a fixed integral subset (see Definition \ref{defn_integral}). If $X$ is a hypersurface in $\mathbb{A}^{n+1}$ defined by a very general polynomial with support $A$, then $X$ is an integral hypersurface containing the origin $0$ and we have
\begin{equation}
\label{eqn_hypersurface_mld}
\mld(0;X)\geq \min \{\sum_{j=1}^{n+1} (\alpha_j-1)+1-\underset{1\leq i\leq N}{\min} \{I^i\cdot \alpha \}+\underset{1\leq j\leq n+1\textnormal{ with }  I_j^i>0}{\min_{1\leq i\leq N}} \{I^i\cdot \alpha-\alpha_j \} \}+n,
\end{equation}
where the first minimum is taken over all $(n+1)$-tuples $\alpha$ such that $\underset{1\leq i\leq N}{\min} \{I^i\cdot \alpha \}$ is attained by at least two different $i$'s.

Moreover, assume the first minimum is attained at some feasible $\alpha$. If for this $\alpha$, the polynomials $P_0$ (defined in equation (\ref{eqn_P0})) and $T_0$ (defined in equation (\ref{eqn_16})) satisfy $\V(P_0)\backslash \V(T_0)\neq \emptyset$ in the torus
\begin{equation*}
(\mathbb{C}^\ast)^{n+1}= \Spec\ \mathbb{C}[x_1^{({\alpha_1})},x_2^{({\alpha_2})},\ldots, x_{n+1}^{({\alpha_{n+1}})}]_{(x_1^{({\alpha_1})},x_2^{({\alpha_2})},\ldots, x_{n+1}^{({\alpha_{n+1}})})},
\end{equation*}
then the inequality (\ref{eqn_hypersurface_mld}) is in fact an equality.
\end{cor}

\subsection{Examples}
$\\$
There are many interesting examples of hypersurfaces where the inequality in Theorem \ref{thm_hypersurface} turns out to be an equality. According to Theorem \ref{thm_hypersurface}, we just need to show that the coefficients are in $\bigcap_\alpha F_\alpha\cap F(A)$, and that $\V(P_0)\backslash \V(T_0)\neq \emptyset$ for certain feasible $\alpha$. In these cases, the invariants $\lambda$ and Mather mld are independent of the coefficients in the defining equations.

\begin{exmp}
\label{exmp_binomial}
Let $X=\mathbb{V}(f)\subset \mathbb{A}^{n+1}$ be an integral variety of dimension $n$ where $f$ is a binomial. Note that $X$ is not necessarily normal. So it might not be a toric variety. The irreducibility of $X$ implies that we can write $f$ in the form
\begin{equation*}
f=ax_1^{\beta_1}x_2^{\beta_2}\cdots x_p^{\beta_p}-bx_{p+1}^{\beta_{p+1}}x_{p+2}^{\beta_{p+2}}\cdots x_{p+q}^{\beta_{p+q}},
\end{equation*}
where $p+q\leq n+1$. If $p+q<n+1$, $X$ is the product of a lower dimensional binomial hypersurface with an affine space. The question is hence reduced to the case when $p+q=n+1$. By assuming $0\in X$, we also require that $p\geq q\geq 1$. The support $A$ contains $N=2$ elements $(\beta_1,\beta_2,\ldots, \beta_p, 0,\ldots, 0)$ and $(0,\ldots, 0,\beta_{p+1},\ldots, \beta_{p+q})$. By requiring that $A$ is integral (see Definition \ref{defn_integral}), we further assume that the line segment connecting these two points does not contain any other integral point. Hence $X$ is integral if the coefficients $(a,b)\in F(A)$ according to Corollary \ref{cor_integral}. On the other hand, by applying a coordinate change that takes $x_1$ to $c\cdot x_1$ and preserves all $x_2,\ldots, x_{p+q}$, we see that any two such hypersurfaces $X$ and $X'$, with different coefficients $(a,b)$ and $(a',b')$, are isomorphic. Hence, we conclude that $F(A)=(\mathbb{C}^\ast)^2$. Clearly, an $(n+1)$-tuple $\alpha\in (\mathbb{Z}_+)^{n+1}$ is feasible (see Remark \ref{rem_feasible}) if and only if
\begin{equation}
\label{eqn_13}
\sum_{i=1}^p \alpha_i\beta_i=\sum_{i=p+1}^{p+q} \alpha_i\beta_i.
\end{equation}
For any feasible $\alpha$, following the notation in the previous section, we have $n_0(\alpha)=\sum_{i=1}^p \alpha_i\beta_i$ and
\begin{equation*}
P_0(x_1^{({\alpha_1})},x_2^{({\alpha_2})},\ldots,x_{n+1}^{({\alpha_{n+1}})}) =f(x_1^{({\alpha_1})},x_2^{({\alpha_2})},\ldots,x_{n+1}^{({\alpha_{n+1}})})
\end{equation*}
is of weight $n_0(\alpha)$. Since $\frac{\partial P_0}{\partial x_j^{(\alpha_j)}}$ is a monomial for each $j$, $F_\alpha=(\mathbb{C}^\ast)^{2}$ for each feasible $\alpha$.

Now fix a feasible $\alpha$. Clearly if $\alpha_{j_0}=\underset{1\leq j\leq n+1}{\max}\{\alpha_j\}$, then we get
\begin{equation*}
T_0(x_1^{({\alpha_1})},x_2^{({\alpha_2})},\ldots,x_{n+1}^{({\alpha_{n+1}})}) =\frac{\partial f(x_1^{({\alpha_1})},x_2^{({\alpha_2})},\ldots,x_{n+1}^{({\alpha_{n+1}})})} {\partial x_{j_0}^{(\alpha_{j_0})}}.
\end{equation*}
In particular, $P_0$ is binomial while $T_0$ is a monomial. Hence $\V(P_0)\backslash \V (T_0)\neq \emptyset$ in the torus $(\mathbb{C}^\ast)^{n+1}$. According to Remark \ref{rem_equality}, for each feasible $\alpha$, we have
\begin{eqnarray*}
\dim (C^m_\alpha)
&=& mn-\sum_{j=1}^{n+1} (\alpha_j-1)-1+ \min_{1\leq i\leq N} \{I^i\cdot \alpha \}-\underset{1\leq j\leq n+1\textnormal{ with }  I_j^i>0}{\min_{1\leq i\leq N}} \{I^i\cdot \alpha-\alpha_j \}\\
&=& mn -\sum_{j=1}^{n+1} (\alpha_j-1) +\underset{1\leq i\leq n+1}{\max}\{\alpha_i-1\}.
\end{eqnarray*}
Therefore, Theorem \ref{thm_hypersurface} implies that
\begin{equation}
\label{eqn_binomial}
\lambda=\min\{\sum_{i=1}^{n+1}\alpha_i -\underset{1\leq i\leq n+1}{\max}\alpha_i-n\},
\end{equation}
or equivalently,
\begin{equation}
\mld(0;X)=\min\{\sum_{i=1}^{n+1}\alpha_i -\underset{1\leq i\leq n+1}{\max}\alpha_i\},
\end{equation}
where the first minimum is taken over all $\alpha\in (\mathbb{Z}_+)^{n+1}$ that satisfy equation (\ref{eqn_13}).
\end{exmp}

\begin{exmp}
Consider the Whitney Umbrella $X=\mathbb{V}(x^2-y^2z)$. The nonsingular locus of $X$ has codimension $1$. Therefore it does not follow in the framework discussed in Section \ref{sec4}, since it is not normal. Nevertheless, we can use the formula (\ref{eqn_binomial}) and conclude that $\lambda=1$ and $\mld(0;X)=3$.
\end{exmp}

\begin{rem}
The binomial hypersurfaces are nice examples where $\lambda$ and Mather mld can be computed directly in a simple form. Note that the result is independent of coefficients $a$ and $b$. This makes sense because we have seen that any two binomial polynomials with the same support define isomorphic hypersurfaces. However, this is not the case if $f$ is more complicated, and then $\lambda$ indeed depends on the coefficients.
\end{rem}

\begin{exmp}
Let $X$ be a curve in $\mathbb{A}^2$ defined by $f=a_1 x^2 +a_2 y^2 +a_3 xy +a_4 y^3$. For a very general choice of coefficients $a_i$, $\lambda$ has a lower bound given by equation (\ref{eqn_hypersurface}). The lower bound is $0$, which is achieved when $\alpha_1=\alpha_2=1$.

First, assume all $a_i$ are equal to $1$. Then $X$ is integral. For any choice of feasible $\alpha$ (see Remark \ref{rem_feasible}), it's clear that $P_0(x,y)$ can only be $x^2+y^2+xy$. Thus the condition $\Delta^\alpha$ (see Condition \ref{open_condition}) is satisfied, or equivalently, we have $(1,1,1,1)\in\cap_\alpha F_\alpha$. Since $T_0(x,y)=2x$ or $2y$, we get $\V(P_0)\backslash \V(T_0)\neq \emptyset$ in the torus $(\mathbb{C}^\ast)^4$. By Theorem \ref{thm_hypersurface}, we conclude that $\lambda=0$, and the minimum in equation (\ref{eqn_hypersurface}) is attained at the tuple $\alpha$ with $\alpha_1=\alpha_2=1$.

Now instead we assume that $a_1=a_2=a_4=1$ and $a_3=2$. $X$ is still integral. But condition $\Delta^\alpha$ is no longer satisfied. By computing $\dim(\psi_m(\pi^{-1}(0)))$ directly from definition, it can be shown that $\lambda=1$.
\end{exmp}

\begin{exmp}
\label{exmp_2}
Let $X\subset \mathbb{A}^{n+1}$ be a hypersurface defined by $f=\sum_{i=1}^{n+1} x_i^{b_i}$, with $n\geq 2$. As Lemma \ref{lem_second} suggests, we consider only feasible $(n+1)$-tuples $\alpha$ (see Remark \ref{rem_feasible}). Clearly, $\alpha$ is feasible if and only if $\underset{1\leq i\leq n+1}{\min} \{ b_i\alpha_i\}$ is attained by at least two different $i$'s.

Note that for any feasible $\alpha$, $\frac{\partial P_0}{\partial x_j^{(\alpha_j)}}$ is always a monomial for each $j$. Thus, we have $\cap_{\textnormal{feasible }\alpha} F_\alpha=(\mathbb{C}^\ast)^{n+1}$. Clearly, when $n\geq 2$, $X$ is an integral hypersurface.

Similarly, $T_0$ is always a monomial for any feasible $\alpha$. So we conclude $\mathbb{V}(P_0)\backslash \mathbb{V}(T_0)\neq \emptyset$. According to Theorem \ref{thm_hypersurface}, we get
\begin{eqnarray}
\label{eqn_15}
\lambda &=& \min\{ \sum_{i=1}^{n+1} (\alpha_i-1) +1- \underset{1\leq i\leq n+1}{\min} \{ b_i\alpha_i\} + \underset{1\leq i\leq n+1}{\min} \{(b_i-1)\alpha_i\}\}\textnormal{, or}
\end{eqnarray}
\begin{eqnarray*}
\mld(0;X) &=& \min\{ \sum_{i=1}^{n+1} (\alpha_i-1) +1- \underset{1\leq i\leq n+1}{\min} \{ b_i\alpha_i\} + \underset{1\leq i\leq n+1}{\min} \{(b_i-1)\alpha_i\}\}+n,
\end{eqnarray*}
where the first minimum is taken over all feasible $\alpha$.
\end{exmp}

There are many classical examples that fall into the category of Example \ref{exmp_2}. A large portion of the following class of examples are of this type.

\begin{exmp}
Consider here the ADE singularities. All the varieties here are integral.

\item (1) Singularities of type $A_k$: $X$ is defined by $f=x_1^{k+1}+x_2^2+\cdots x_n^2$ for $n\geq 3$.

    Choose multi-index $\alpha$ with $\alpha_i=1$ for $1\leq i\leq n$. The minimum weight $n_0(\alpha)$ is attained by $n-1$ monomials if $k>1$, or $n$ monomials if $k=1$. In both cases, $\alpha$ is feasible. Let $b_1=k+1$ and $b_i=2$ for $2\leq i\leq n$. Then we have
    \begin{equation*}
    \sum_{i=1}^{n+1} (\alpha_i-1) +1- \underset{1\leq i\leq n+1}{\min} \{ b_i\alpha_i\} + \underset{1\leq i\leq n+1}{\min} \{(b_i-1)\alpha_i\}=0.
    \end{equation*}
Hence according to equation (\ref{eqn_15}), we get $\lambda=0$ or $\mld(0;X)=n-1$.

\item (2) Singularities of type $D_k$: $X$ is defined by $f=x_1^{k-1}+x_1x_2^2 +x_3^2+\cdots + x_n^2$ with $k\geq 4$. One checks easily that the coefficients are in $\cap_\alpha F_\alpha$.

    If $n\geq 4$, then there are at least two quadratic terms. Hence $\alpha=(1,\ldots,1)$ is feasible, which achieves the minimum $0$ in equation (\ref{eqn_hypersurface}). Note that we have
\begin{equation*}
P_0=(x_3^{(1)})^2+\ldots+(x_n^{(1)})^2
\end{equation*}
and $T_0=\frac{\partial P_0}{\partial x_3^{(1)}}=2x_3^{(1)}$. Therefore, $\mathbb{V}(P_0)\backslash \mathbb{V}(T_0)\neq \emptyset$. By Theorem \ref{thm_hypersurface}, we get $\lambda=0$ or $\mld(0;X)=n-1$.

    When $n=3$, the minimum $1$ of equation (\ref{eqn_hypersurface}) is achieved when $\alpha=(2,1,2)$. With similar analysis, we obtain $\lambda=1$ or $\mld(0;X)=n=3$.

\item (3) Singularities of type $E_6$: $X$ is defined by $f=x_1^4+x_2^3+x_3^2+\ldots+x_n^2$. This belongs to Example \ref{exmp_2}. So we use equation (\ref{eqn_15}).

    If $n\geq 4$, with $\alpha=(1,\ldots,1)$, we get $\lambda=0$ or $\mld(0;X)=n-1$. When $n=3$, the minimum is achieved when $\alpha=(1,2,2)$, and we get $\lambda=1$ or $\mld(0;X)=n=3$.

\item (4) Singularities of type $E_7$: $X$ is defined by $f=x_1^3x_2+x_2^3+x_3^2+\ldots+x_n^2$.

This is very similar to case (2). Again one checks easily that the coefficients satisfy Condition $\Delta^\alpha$ for all feasible multi-indices $\alpha$.

    If $n\geq 4$, $\alpha=(1,\ldots,1)$ is feasible. Similar to (2) we get $\lambda=0$ or $\mld(0;X)=n-1$.

    When $n=3$, $\alpha=(2,2,3)$ is feasible and it gives minimum in equation (\ref{eqn_hypersurface}). Simple analysis similar to the ones above shows that we have an equality and hence $\lambda=2$ or $\mld(0;X)=n+1=4$.

\item (5) Singularities of type $E_8$: $X$ is defined by $f=x_1^5+x_2^3+x_3^2+\ldots+x_n^2$. This belongs to Example \ref{exmp_2} so we can apply formula (\ref{eqn_15}).

    When $n\geq 4$, we get $\lambda=0$ or $\mld(0;X)=n-1$. The minimum is attained when $\alpha=(1,\ldots,1)$. When $n=3$, we have $\lambda=2$ or $\mld(0;X)=n+1=4$, and it is attained when $\alpha=(2,2,3)$.
\end{exmp}

\subsection{Possible generalizations}
$\\$
We only treat the case when the hypersurface is defined by a very general polynomial with a fixed support. An obvious question is: what can we say if the hypersurface is defined by a general polynomial (so that it is integral) with a fixed support? Unfortunately, the polynomials $P_0$ defined in equation (\ref{eqn_P0}) and $T_0$ defined in equation (\ref{eqn_16}) don't behave well and our method fails.

An obvious generalization of the results in this section is to treat the class of complete intersection varieties. However, our method doesn't work well when there are multiple defining equations.


\begin{thebibliography}{EMY02}

\bibitem[Amb06]{Am06}
Florin Ambro.
\newblock The minimal log discrepancy.
\newblock In {\em Proceedings of the Symposium ``Multiplier ideals and arc
  spaces''(RIMS 2006), K. Watanabe (Ed.), RIMS Koukyuuroku}, volume 1550, pages
  121--130, 2006.

\bibitem[dF16]{dF16}
Tommaso de~Fernex.
\newblock The space of arcs of an algebraic variety.
\newblock {\em arXiv preprint arXiv:1604.02728}, 2016.

\bibitem[dFD11]{dFD11}
Tommaso de~Fernex and Roi Docampo.
\newblock Jacobian discrepancies and rational singularities. to appear in j.
  eur. math. soc.
\newblock {\em arXiv preprint arXiv:1106.2172}, 2011.

\bibitem[dFEI07]{dFEI08}
Tommaso de~Fernex, Lawrence Ein, and Shihoko Ishii.
\newblock Divisorial valuations via arcs.
\newblock {\em Publ. RIMS, 44, (2008) 425-448}, 2007.

\bibitem[dFT16]{dFT16}
Tommaso de~Fernex and Yu-Chao Tu.
\newblock Towards a link theoretic characterization of smoothness.
\newblock {\em arXiv preprint arXiv:1608.08510}, 2016.

\bibitem[DL99]{DL99}
Jan Denef and Fran{\c{c}}ois Loeser.
\newblock Germs of arcs on singular algebraic varieties and motivic
  integration.
\newblock {\em Inventiones mathematicae}, 135(1):201--232, 1999.

\bibitem[EI15]{EI13}
Lawrence Ein and Shihoko Ishii.
\newblock Singularities with respect to mather--jacobian discrepancies.
\newblock {\em Commutative Algebra and Noncommutative Algebraic Geometry},
  2:125, 2015.

\bibitem[EM05]{EM05}
L~Ein and M~Mustata.
\newblock Jet schemes and singularities, algebraic geometry--seattle 2005, part
  2, 505--546.
\newblock In {\em Proc. Sympos. Pure Math}, volume~80, 2005.

\bibitem[EM09]{EM09}
Lawrence Ein and Mircea Mustata.
\newblock Jet schemes and singularities.
\newblock In {\em Proc. Sympos. Pure Math}, volume 80, Part 2, pages 505--546,
  2009.

\bibitem[EMY02]{EMY03}
Lawrence Ein, Mircea Mustata, and Takehiko Yasuda.
\newblock Jet schemes, log discrepancies and inversion of adjunction.
\newblock {\em Invent. Math. 153 (2003), 119-135.}, 2002.

\bibitem[Ful93]{Ful93}
William Fulton.
\newblock {\em Introduction to toric varieties}.
\newblock Number 131 in Annals of Mathematics Studies. Princeton University
  Press, 1993.

\bibitem[IR13]{Ish13}
Shihoko Ishii and Ana~J Reguera.
\newblock Singularities with the highest mather minimal log discrepancy.
\newblock {\em Mathematische Zeitschrift}, 275(3-4):1255--1274, 2013.

\bibitem[IR15]{Ish15}
Shihoko Ishii and Ana Reguera.
\newblock Singularities in arbitrary characteristic via jet schemes.
\newblock {\em arXiv preprint arXiv:1510.05210}, 2015.

\bibitem[Ish04]{Ish03}
Shihoko Ishii.
\newblock The arc space of a toric variety.
\newblock {\em Journal of Algebra}, 278(2):666--683, 2004.

\bibitem[Ish13]{Ish11}
Shihoko Ishii.
\newblock Mather discrepancy and the arc spaces.
\newblock In {\em Annales de l'institut Fourier}, volume 63, Part 1, pages
  89--111. Association des Annales de l'institut Fourier, 2013.

\bibitem[Mou11]{Mou11}
Hussein Mourtada.
\newblock Jet schemes of toric surfaces.
\newblock {\em Comptes Rendus Mathematique}, 349(9):563--566, 2011.

\bibitem[Mus14]{M14}
Mircea Mustata.
\newblock The dimension of jet schemes of singular varieties.
\newblock {\em arXiv preprint arXiv:1404.7731}, 2014.

\bibitem[Sho04]{Sh04}
Vyacheslav~Vladimirovich Shokurov.
\newblock Letters of a bi-rationalist v: Mld's and termination of log flips.
\newblock {\em Trudy Matematicheskogo Instituta imeni VA Steklova},
  246:328--351, 2004.

\bibitem[Stu95]{St95}
Bernd Sturmfels.
\newblock Equations defining toric varieties. algebraic geometry?santa cruz
  1995, 437--449.
\newblock In {\em Proc. Sympos. Pure Math}, volume~62, 1995.

\bibitem[Yu16]{Yu16}
Josephine Yu.
\newblock Do most polynomials generate a prime ideal?
\newblock {\em Journal of Algebra}, 459:468--474, 2016.

\end{thebibliography}

\end{document}